\documentclass[amsfonts,11pt]{amsart}
\RequirePackage{amsmath,amssymb,amsthm, amscd, comment,mathtools}
\usepackage[margin=1.2in]{geometry}
\usepackage{amsmath}
\usepackage{amsthm}
\usepackage{amssymb}
\usepackage{amscd}
\usepackage{mathrsfs}
\usepackage{graphicx}
\usepackage{rotating, graphicx}
\usepackage[all,cmtip]{xy}
\usepackage{enumerate}
\usepackage{tabularx}
\usepackage{tikz-cd}
\usepackage[thinlines]{easytable}
\usepackage{multirow}
\usepackage{mathtools}
\usepackage{mathrsfs}
\usepackage{changepage} 
\usepackage{paralist}

\usepackage{kantlipsum}
\usepackage{mathabx}
\usepackage{enumitem}
\usepackage{times}
\usepackage{subfig}
\usepackage{color}
\usepackage[pagebackref,breaklinks,colorlinks,linkcolor=red,anchorcolor=red,citecolor=blue]{hyperref}

\calclayout

\numberwithin{equation}{section}

\newtheorem{proposition}[subsubsection]{Proposition}
\newtheorem{corollary}[subsubsection]{Corollary}
\newtheorem{theorem}[subsubsection]{Theorem}

\theoremstyle{definition}
\newtheorem{definition}[subsubsection]{Definition}
\newtheorem{lemma}[subsubsection]{Lemma}
\theoremstyle{remark}
\newtheorem{remark}[subsubsection]{Remark}

\DeclareMathAlphabet{\mathbbold}{U}{bbold}{m}{n}
\numberwithin{equation}{subsubsection}

\newcommand*{\id}{\mathrm{id}}

\newcommand*{\Hom}{\operatorname{Hom}}

\newcommand*{\Z}{\mathbb{Z}}
\newcommand*{\sign}{\mathrm{Sign}}

\newcommand*{\QQ}{\mathbb{Q}}

\newcommand*{\sheaf}[1]{{\mathcal #1}}
\newcommand*{\Oo}{\sheaf{O}}
\newcommand*{\Spec}{\mathrm{Spec}}

\newcommand{\can}{\mathrm{can}}

\newcommand*{\I}{\mathrm{I}}                

\newcommand{\on}{ ~\mathrm{on}~ }
\newcommand{\D}{D}

\newcommand{\op}{\operatorname}

\newcommand{\R}{\mathbb{R}}

\DeclareMathOperator{\support}{supp}

\newcommand{\Sper}{\textnormal{Sper}} 

\newcommand{\Q}{\mathcal{Q}}
\newcommand{\km}{k_\mathfrak{m}}
\newcommand{\W}{\mathrm{W}}
\newcolumntype{M}[1]{>{\centering\arraybackslash}m{#1}}
\newcolumntype{N}{@{}m{0pt}@{}}

\newcommand{\xdownarrow}[1]{%
	{\left\downarrow\vbox to #1{}\right.\kern-\nulldelimiterspace}
}
\newcommand{\xuparrow}[1]{%
	{\left \uparrow\vbox to #1{}\right.\kern-\nulldelimiterspace}
}
\makeatletter
\newcommand{\extp}{\@ifnextchar^\@extp{\@extp^{\,}}}
\def\@extp^#1{\mathop{\bigwedge\nolimits^{\!#1}}}
\makeatother

\newcommand{\kmm}{k(\mathfrak{m})}
\newcommand{\kmmt}{k(\tilde{\mathfrak{m}})}
\newcommand{\mmt}{\tilde{\mathfrak{m}}}

\title{A Gersten complex on real schemes}

\author{Fangzhou Jin}
\address{Fangzhou Jin, School of Mathematical Sciences\\
Tongji University\\
Siping Road 1239\\
200092 Shanghai\\
P. R. China}
\email{\href{mailto:fangzhoujin@tongji.edu.cn}{fangzhoujin@tongji.edu.cn}}

\author{Heng Xie}
\address{Heng Xie, School of Mathematics\\
Sun Yat-sen University, 
510275 Guangzhou\\
P. R. China}
\email{\href{xieh59@math.sysu.edu.cn}{xieh59@math.sysu.edu.cn}}

\date{\today}

\begin{document}

\maketitle

\begin{abstract}
We discuss a connection between coherent duality and Verdier duality via a Gersten-type complex of sheaves on real schemes, and show that this construction gives a dualizing object in the derived category, which is compatible with the exceptional inverse image functor $f^!$. The hypercohomology of this complex coincides with hypercohomology of the sheafified Gersten-Witt complex, which in some cases can be related to topological or semialgebraic Borel-Moore homology. 
\end{abstract}

\tableofcontents

\section{Introduction}

In algebraic geometry and topology, there are several frameworks in which local and global duality theorems can be incarnated by a machinery of \emph{Grothendieck six functors formalism}, such as
\begin{itemize}
\item [(D1)] coherent duality in the derived category of quasi-coherent sheaves over a scheme;
\item [(D2)] \'etale duality in the derived category of torsion \'etale sheaves over a scheme;
\item [(D3)] Verdier duality in the derived category of sheaves over a locally compact topological space;
\item [(D4)] $\mathbb{A}^1$-homotopic duality in triangulated categories of motivic sheaves.
\end{itemize}

All these formalisms are modelled over (D1) in \cite{Har}, where one of the main upshot is the existence of \emph{dualizing objects} and the fact that they are preserved by the exceptional inverse image functor $f^!$. However, (D1) behaves quite differently from the others: for example, in the coherent context, the exceptional direct image functor $f_!$ does not exist, and there is no suitable subcategory of \emph{constructible objects} perserved by the six functors, while base change and purity theorems hold in a much greater generality. One of the main goals of this paper is to build up a bridge between (D1) and a variant of (D3) related to semialgebraic spaces, which is provided by a Gersten complex of \textit{real schemes}. In this new way, dualizing complexes in (D1) naturally give rise to some explicit dualizing objects in (D3). 

From a more concrete point of view, we are motivated by the following question:
\begin{center}
Is there a natural analog of singular cohomology in algebraic geometry?
\end{center} 
Classically, for smooth varieties $X$ over  $\mathbb{C}$, there are \emph{cycle class maps} 
\begin{align}
\op{CH}^k(X)\to \op{H}^{2k}(X(\mathbb{C}),\Z)
\end{align}
from Chow groups to singular cohomology.\ 
For an algebraic variety $X$ over the field of real numbers $\R$, we have the following analogue of the cycle class maps: if we denote by $\I^j$ the sheaf of $j$-th power of the fundamental ideal in the Witt group, 
then there are \emph{twisted signature morphisms} from twisted $\I^j$-cohomology to twisted singular cohomology with local coefficients on $X(\R)$
\begin{align}
\label{eq:signtwi}
\sign: H^i(X,\I^j(L)) \to H^i(X(\R),\Z(L))
\end{align}
for any line bundle $L$ on $X$, which are compatible with pullbacks, pushforwards and intersection products, and the map~\eqref{eq:signtwi} is an isomorphism if $j \geq \dim X +1 $ (see \cite{jacobson} for the untwisted case, and \cite{HWXZ} for the general case).
\footnote{In \cite{HWXZ}, it is also shown that if $X$ is a smooth cellular variety, then the map~\eqref{eq:signtwi} is always an isomorphism for any $j$.}
Here the ``twists'' by a line bundle $L$ are defined \`a la Morel (see~\ref{eq:twimorel} below), which can also be understood on the Witt-theoretic side, see \cite[2.32]{HWXZ}.
\footnote{Such twists do not appear in the case for varieties over $\mathbb{C}$, as complex vector bundles are always canonically oriented, which is not the case for real vector bundles.}



The intrinsic relation between the $\I^j$-cohomology and real algebraic geometry can be better understood via the theory of \emph{real schemes}: for any scheme $X$, the associated real scheme $X_r$ is a topological space whose underlying sets consists of points of $X$ together with orderings on the residue fields (see~\ref{num:rsch} below), which is a \emph{spectral space} in the sense of Hochster (\cite{Hoc}). The definition of the real scheme stems from the study of the link between real algebraic geometry and \'etale cohomology with $2$-torsion coefficients (\cite{Sch}), and recent progress provides more interesting relations of this object with other areas of mathematics: for example, as a variant of the map~\eqref{eq:signtwi}, Jacobson (\cite[Theorem 8.6]{jacobson}) constructs an isomorphism relating the sheaf cohomology of the powers of the fundamental ideal in the Witt groups of quadratic forms with the constant sheaf cohomology on the real scheme 
\begin{align}
\label{eq:jacreal}
H^n(X,\I^\infty)\simeq H^n(X_r,\mathbb{Z})
\end{align}
(see~\ref{num:jac8.6} below for a generalization); more recently, Bachmann shows that the minus part of the motivic stable homotopy category is precisely provided by the information on the real scheme (\cite[Theorem 35]{Bac}). 

On the other hand, within the framework of the six functors in motivic homotopy, a systematic theory that relates cohomology and \emph{Borel-Moore theories} (also called \emph{bivariant theory}) is developed by D\'eglise (\cite{Deg}), based on some ideas of Voevodsky. The table below indicates how these two types of theories can be philosophically compared (see also \cite{DJK}):
\begin{center}
\begin{tabular}{ |c|c| } 
 \hline
 cohomology & Borel-Moore theory  \\ 
 $\otimes$-invertible objects & dualizing objects  \\ 
 preserved by $f^*$ & preserved by $f^!$  \\ 
 ring structure & module over cohomology  \\ 
 \hline
\end{tabular}
\end{center}
For example, algebraic $G$-theory agrees with the Borel-Moore theory associated to algebraic $K$-theory (\cite{Jin1}). For smooth schemes, cohomology agree with Borel-Moore theory up to a twist by Poincar\'e duality. In general, cohomology theories usually have a ring structure, while Borel-Moore theories have better localization properties.

From this point of view, Jacobson's isomorphism~\eqref{eq:jacreal} can be regarded as an identification of cohomology theories. Since Bachmann's theorem relates sheaves on real schemes to motivic homotopy, a natural question arises whether there is a Borel-Moore type analogue . Just as Grothendieck's coherent duality (\cite{Har}) refines Serre's duality by considering \emph{dualizing complexes} instead of \emph{dualizing sheaves}, it is attempting to think that such an analogue for real schemes should involve complexes of sheaves. Following this vein, one of the first results of this paper shows that, the isomorphism~\eqref{eq:signtwi} naturally generalizes to singular varities, and provide the following signature isomorphism
\begin{align}
\sign: H_{n} (X, \I^\infty(L)) \simeq H^{\mathrm{BM}}_n(X(\R), \Z(L))
\end{align}
where the right side is the classical Borel-Moore homology in topology (see Corollary \ref{coro:signature-borel-moore} below).

On the Witt-theoretic side, a candidate of such a complex is constructed by Gille (\cite[7.1]{Gil07}), which associates to a (coherent) dualizing complex a Gersten-type complex of Witt groups. We will reproduce this construction in Section~\ref{sec:BWGcomplex}, while avoiding the choices of injective hulls as in \cite{Gil07}. It is worth mentioning that a very similar construction using residual complexes appears recently in \cite[\S 3]{Plowman}. For a discussion for the corresponding Gersten conjecture, see Remark~\ref{rmk:gerstenconj}. 

On the side of real schemes, our inspiration comes from a complex of sheaves defined by Scheiderer (\cite[Corollary 2.3]{Sch1}): for a regular scheme, Scheiderer's result states that the constant sheaf over the real scheme has a canonical resolution by acyclic sheaves which comes from the $E_1$-page of the coniveau spectral sequence. Since the structure sheaf of a regular scheme is a dualizing complex, by comparing with Gille's complex, one tend to believe that a natural analogue on real schemes could be given by a complex of acyclic sheaves associated to a (coherent) dualizing complex.

These considerations lead us to define, for a scheme $X$ with a dualizing complex $K$, a new complex $\underline{G}(X_r,K^{ })$ of sheaves of abelian groups on $X_r$, as a Rost-Schmid type complex where the residue homomorphisms are given by the information captured from the stalks of the dualizing complex, using the twisting techniques from \cite{HWXZ}. Such a construction is very natural to us, and maybe inherent to Grothendieck's insights on ``residues and duality''. 

\subsubsection*{Acknowledgments}
We would like to thank Fr\'ed\'eric D\'eglise, Jean Fasel, Niels Feld, Adeel Khan, James Plowman, Marco Schlichting and Marcus Zibrowius for helpful discussions. FJ learned the difference between various versions of the six functors formalism from lectures on zeta values and L-functions of motives given by Bruno Kahn at Universit\'e Paris 6.

Both authors are partially supported by the DFG Priority Programme SPP 1786. FJ acknowledges support from the National Key Research and Development Program of China Grant Nr.2021YFA1001400, the National Natural Science Foundation of China Grant Nr.12101455, the Fundamental Research Funds for the Central Universities, and the ERC Project-QUADAG, which has received funding from the European Research Council (ERC) under the European Union’s Horizon 2020 research and innovation programme (grant agreement Nr.832833). HX is supported by the Fundamental Research Funds from the Central Universities, Sun Yat-sen University 34000-31610294, NSFC Grant 12271529, NSFC Grant 12271500, EPSRC Grant EP/M001113/1 and he would also like to thank Max Planck Institute for Mathematics in Bonn for its hospitality. HX would like to acknowledge the DFG-funded research training group GRK 2240: Algebro-Geometric Methods in Algebra, Arithmetic and Topology.

\section{Conventions and preliminaries}

\subsection{On schemes and sheaves}
\label{sec:conventionsch}

\subsubsection{}
A \textbf{precomplex} (of abelian groups) stands for a sequence of maps of abelian groups.

\subsubsection{}
All schemes considered are noetherian schemes of finite Krull dimension. Smooth morphisms are seperated of finite type. If $X$ is a scheme and $x$ is a point of $X$, denote by $k(x)$ the residue field of $x$. 

\subsubsection{}
All sheaves considered are sheaves of abelian groups. We identify a vector bundle with the locally free sheaf given by the sheaf of its sections.

\subsubsection{}
\label{sec:outtens}
	If $X\to S$ and $Y\to S$ are two morphisms of schemes, denote by $p_X:X\times_SY\to X$ and $p_Y:X\times_SY\to Y$ the projections. If $\mathcal{F}$ (respectively $\mathcal{G}$) is a sheaf of abelian groups on $X$ (respectively $Y$), denote by $\mathcal{F}\boxtimes_S\mathcal{G}:=p_X^*\mathcal{F}\otimes p_Y^*\mathcal{G}$.

\subsubsection{}
\label{eq:twimorel}
Let $(X,\mathcal{O}_X)$ be a ringed space and let $ L$ be an invertible $\mathcal{O}_X$-module. If $\mathcal{F}$ is a sheaf on $X$ with an action of the sheaf $\mathcal{O}_X^\times$, we define the \textbf{twisted sheaf} $\mathcal{F}( L)$ to be the sheaf associated to the presheaf
	\begin{align}
	U\mapsto \mathcal{F}(U)\otimes_{\mathbb{Z}[\mathcal{O}_X^\times(U)]}\mathbb{Z}[ L(U)^\times]
	\end{align}
	where $\mathbb{Z}[ L(U)^\times]$ denotes the free abelian group generated by elements of $ L(U)^\times$ (\cite[\S 5]{Mor}, \cite[Definition 2.28]{HWXZ}). An isomorphism $ L\simeq\mathcal{O}_X$ induces an isomorphism $\mathcal{F}\simeq\mathcal{F}( L)$, so in particular the twisted sheaf $\mathcal{F}( L)$ is always locally isomorphic to $\mathcal{F}$. 
	
	The twisting operation is compatible with pullbacks and tensor products of invertible modules: if $f:Y\to X$ is a morphism, then $f^*(\mathcal{F}( L))=(f^*\mathcal{F})(f^* L)$; if $ L$ and $ L'$ are two invertible $\mathcal{O}_X$-modules, then $(\mathcal{F}( L))( L')=\mathcal{F}( L\otimes L')$.

\subsubsection{}
A quasi-projective morphism $f:Y\to X$ is a \textbf{local complete intersection} (abbreviated as \textbf{lci}) morphism if it factors as a regular closed immersion followed by a smooth morphism. Such a morphism is a perfect morphism, and its cotangent complex $\tau_f$ is a perfect complex of $\mathcal{O}_Y$-modules.

\subsubsection{}
Let $X$ be a regular scheme. For any point $x\in X$, we denote by $\omega_{x/X}$ the inverse of the determinant of the normal bundle of the regular closed immersion $\operatorname{Spec}(k(x)) \to\operatorname{Spec}(\mathcal{O}_{X,x})$. 


\subsection{On dualizing complexes}
\label{num:recalldualcomp}

\subsubsection{}
If $X$ is a scheme, we denote by $D^b_c(\mathcal{Q}(X))$ the derived category of bounded complexes of quasi-coherent sheaves on $X$ with coherent cohomology. 

\subsubsection{}
For $M^\bullet\in D^b_c(\mathcal{Q}(X))$, the \textbf{support} of $M^\bullet$ is the subset of $X$ defined as 
\begin{align} 
	\support M^\bullet=\{ x \in X | \mathcal{H}^i(M^\bullet)_x \neq 0 \textnormal{ for some $i \in \Z$ }  \}.
\end{align}

\subsubsection{}
A \textbf{dualizing complex} on a scheme $X$ is a complex $K^\bullet \in D^b_c(\mathcal{Q}(X))$ 
\begin{align}
K^\bullet= ( \cdots \rightarrow 0 \rightarrow K^m \stackrel{d^m}\longrightarrow K^{m+1} \stackrel{d^{m-1}}\longrightarrow \cdots \stackrel{d^{n-1}}\longrightarrow K^n \rightarrow 0 \rightarrow \cdots)
\end{align} 
such that for any $\mathcal{F}^\bullet\in D^b_c(\mathcal{Q}(X))$ the canonical map
\begin{align}
	\mathcal{F}^\bullet\to R\underline{Hom}_{\mathcal{O}_X}( R\underline{Hom}_{\mathcal{O}_X}(\mathcal{F}^{\bullet},K^{\bullet}),K^{\bullet}) 
\end{align}
	is an isomorphism. If moreover, $K^r$ is injective and an essential extension of $\ker(d^r)$ for all $r \in \mathbb{Z}$, we say $K^\bullet$ is a \textbf{minimal dualizing complex}. We will usually use the simplified notation $K$ without the bullet to mean the complex $K^\bullet$, if no confusion occurs.  
	
	Note that in contrast, a \emph{dualizing object} will refer to the real \'etale duality in this paper, see Definition~\ref{def:dualizing}.


\subsubsection{}
If $f:X\to Y$ is a separated morphism of finite type, there is a functor $f^!:D^+_c(\mathcal{Q}(Y))\to D^+_c(\mathcal{Q}(X))$, which can be constructed using smoothings (\cite[III 8.7]{Har}) or Nagata compactifications (\cite{LH}). 

\subsubsection{}
\label{num:dualunique}
If $X$ is a scheme, any two dualizing complexes on $X$ differ by an invertible sheaf up to a shift (\cite[V 3.1]{Har}). 

\subsubsection{}
\label{num:codimfunction}
If $x$ is a point of $X$ and $K^{ }$ is a dualizing complex on $X$, denote by $\pi_x:\operatorname{Spec}(k(x))\to\operatorname{Spec}(\mathcal{O}_{X,x})$ the canonical map and $K^{ }_x\in D^b_c(\mathcal{Q}(\mathcal{O}_{X,x}))$ the pullback of the complex $K $ via the canonical morphism $\operatorname{Spec}(\mathcal{O}_{X,x}) \to X$. Then there is a unique integer $\mu_{K^{ }}(x)$ such that $\pi_x^!K^{ }_x\in D^b_c(\mathcal{Q}(k(x)))$ is quasi-isomorphic to a $1$-dimensional $k(x)$-vector space concentrated in degree $\mu_{K^{ }}(x)$; in addition, the association $x\mapsto\mu_{K^{ }}(x)$ defines a \emph{codimension function} on $X$ (\cite[V 3.4, 7.1]{Har}). We will use the letters $m$ and $n$ respectively to refer to the minimum and the maximum of $\mu_K$.

\subsubsection{ }
Denote $X^{(p)}: = \{ x \in X | \mu_{K} (x) = p \} $. 

\subsubsection{ }
\label{num:Kkx}
If $K$ is a minimal dualizing complex and $\mu_K(x)=p$, then we have an identification
$$ \pi_x^!K^{ }_x =  K_{k(x)} [-p] $$
where $K_{k(x)}: = \Hom_{\mathcal{O}_{X,x}}(k(x),K^p_x)$ is a one dimensional vector space over $k(x)$.\ This fact follows directly from \cite[Proposition 1.10]{Gil07} and Section \ref{num:codimfunction}. Here, note that $\pi_x^!K_x$ is a minimal dualizing complex over the field $k(x)$. 


\subsubsection{}
\label{num:codlb}
If $ L$ is an invertible $\mathcal{O}_X$-module, then $\mu_{K^{ }}$=$\mu_{K^{ }\otimes  L}$ and $(K^{ }\otimes  L)_{k(x)}\cong K_{k(x)}\otimes L_{k(x)}$. 

\subsubsection{}
\label{num:regcod}
If $X$ is a regular scheme, then there is a minimal dualizing complex $K_X$ resolving $\Oo_X$. Note that $\mu_{K_X}: X\to \Z$ is the ususal codimension function on $X$ (cf. \cite[V \S 10]{Har}). If $L$ is a line bundle on $X$, then we denote $K_L: = K_X \otimes L$ which is the minimal dualizing complex resolving $L$. 

\subsubsection{}
If $f:X\to Y$ is a quasi-projective lci morphism of relative dimension $d$ and $K$ is a dualizing complex on $Y$, then there is an isomorphism in $D^b_c(\mathcal{Q}(X))$
\begin{align}
\label{eq:lcipur}
f^!K\simeq f^*K\otimes \operatorname{det}(\tau_f)[d]
\end{align}
which is compatible with composition of quasi-projective lci morphisms.
If $f$ is a regular closed immersion, this is the \emph{fundamental local isomorphism} (\cite[III 7.3, p. 180]{Har}); if $f$ is smooth, the isomorphism is defined in \cite[III \S2]{Har}. In general, the isomorphism~\eqref{eq:lcipur} is defined by gluing these two cases, and is independent of the factorization of $f$ as a regular closed immersion followed by a smooth morphism (\cite[III \S8]{Har}).


\subsection{On real schemes}

\subsubsection{}
\label{num:rsch}
For a scheme $X$, the associated \textbf{real scheme} $X_r$ is a topological space constructed by gluing real spectra of rings (cf. \cite[0.4.2]{Sch1}).\ Its points are pairs $(x,P)$, where $x$ is a point of $X$ and $P$ is an ordering of the residue field of $x$. Forgetting the orderings on residue fields yields a canonical continuous map
\begin{align}
\operatorname{supp}:X_r\to X.
\end{align}

\subsubsection{}
\label{sec:2invert}


For any scheme $X$, the canonical inclusion $X\times_\mathbb{Z}\mathbb{Q}\to X$ is a pro-open immersion, which induces an isomorphism on the underlying real schemes $X_r=(X\times_\mathbb{Z}\mathbb{Q})_r$, that is, the underlying real scheme only depends on the characteristic $0$ fiber. This is because any field of positive characteristic cannot have an ordering. Therefore in the study of real schemes it is harmless to assume that all schemes have characteristic $0$.

\subsubsection{}
We now introduce the analog of twists \`a la Morel in~\ref{eq:twimorel} for real schemes. 
Recall that for any ordering $P$ on a field $F$, there is a group homomorphism $\sign_P\colon F^\times \rightarrow \{1, -1\}$  sending $a$ to $1$ if $a \in P$ and $a$ to $-1$ if $a \not\in P$.
	\begin{lemma}\label{lem:localpatching}
		For any Zariski open subset $U\subset X$ and regular function \(f\in\sheaf O_X^\times(U)\), the sign map \(\sign(f)\colon U_r \to \{\pm 1\}\) given by \((x,P)\mapsto \sign_P(f_x)\) is locally constant.
	\end{lemma}
	\begin{proof}
		We may assume that $U = \Spec A$ is affine.  The global section $f\colon A \rightarrow A$ is an automorphism of \(A\), viewed as a free module over itself, so we may consider \(f\) as a unit in $A$.  Recall that a subbasis of the topology on the real spectrum $\op{sper}A$ is given by the sets $\op{D}(a):= \{ (x,P) \in \Sper A\colon a_x >_P 0 \}$ for \(a\in A\). As \(f_x\) is non-zero for any \(x\), $\Sper A$ is the disjoint union of $\op{D}(f)$ and $\op{D}(-f)$.  The value of \(\sign_P(f_x)\) is constant \(+1\) on \(\op{D}(f)\) and constant \(-1\) on \(\op{D}(-f)\).  
	\end{proof}
	
	We thus obtain a morphism of abelian sheaves \(\sheaf O_X^\times \to \support_*\{\pm 1\}\), and, by adjunction, a morphism \(\support^*(\sheaf O_X^\times)\to\{\pm 1\}\).  We can upgrade this morphism to a morphism of sheaves of rings \(\Z[\support^*(\sheaf O_X^\times)]\to\Z[\{\pm 1\}]\). This allows us to view any abelian sheaf on \(X_r\) as a \(\Z[\support^*(\sheaf O_X^\times)]\)-module.  
\begin{definition}
\label{def:realtw}
Let \(\sheaf L\) be an invertible $\mathcal{O}_X$-module, and let \(\sheaf A\) be a sheaf of abelian groups on $X_r$.
For any open constructible subset $U$ of $X_r$, we define the group
\begin{align}
C(U,\sheaf A(\sheaf L))=\sheaf A (U) \otimes_{\Z[\support^*(\Oo_X^\times)]} \Z[\support^*(\sheaf L^\times)]
\end{align}
and the \textbf{twisted sheaf} $\sheaf A(\sheaf L)$ is the sheaf of abelian groups on $X_r$ defined as the sheafification of the presheaf $U\mapsto C(U,\sheaf A(\sheaf L))$, which is a sheaf locally isomorphic to $\sheaf A$.
\end{definition}
In particular, for any invertible $\mathcal{O}_X$-module \(\sheaf L\), we denote by $\mathbb{Z}( L)$ the twisted sheaf associated to the constant sheaf $\mathbb{Z}$ on $X_r$ therefore obtained, which is a locally constant sheaf locally isomorphic to $\mathbb{Z}$. This sheaf will play an essential role in what follows.



\subsubsection{}
\label{num:jac8.6}
We denote by $\I^\infty=\operatorname{colim_j}\I^j$ (resp. $\I^\infty=\operatorname{colim_j}\I^j$) the colimit of the powers of the fundamental ideal (resp. the colimit of the Zariski sheaves of the powers of the fundamental ideal). By \cite[Appendix]{HWXZ}, for any invertible $\mathcal{O}_X$-module $ L$ on $X$, there is an isomorphism of sheaves
\begin{align}
\label{eq:jac8.6intro}
\I^\infty( L)\simeq \support_*(\mathbb{Z}( L)).
\end{align}
In the case $ L=\mathcal{O}_X$, the isomorphism~\eqref{eq:jac8.6intro} is due to Jacobson (\cite[Theorem 8.6]{jacobson}).




\section{Main results of the paper}
For a scheme $X$ with a minimal dualizing complex $K$, we recall in Section~\ref{sec:BWGcomplex} the construction of the Gersten-Witt complex. In Section~\ref{sect:Rost-Schmid-complex} we construct in a different way a Rost-Schmid complex of the form
\begin{align} 
\label{eq:Rost-Schmid-complex-Wittintro}
\bigoplus\limits_{\mu_K(x)=m} \W(k(x), K_{k(x)}) 
\to
\bigoplus\limits_{\mu_K(x)=m+1} \W(k(x), K_{k(x)})
\to\cdots\to
\bigoplus\limits_{\mu_K(x)=n} \W(k(x), K_{k(x)}).
\end{align}
The construction follows the ideas of \cite{Schmid}, using the twisted residue and transfer maps on Witt groups. When $2$ is invertible on $X$, the precomplex~\eqref{eq:Rost-Schmid-complex-Wittintro} agrees with the Gersten-Witt complex, and in particular is a complex (cf. Theorem~\ref{theorem:Gille-Schmid}). 

In Sections~\ref{sec:Icohdual} and~\ref{sec:Icohsupp} we study the restriction of the complex~\eqref{eq:Rost-Schmid-complex-Wittintro} to the powers of the fundamental ideal, as well as its cohomology with support. The main upshot is the following devissage-type result (cf.\  Theorem~\ref{thm:devissagK_Ij}):
\begin{theorem}
	\label{thm:devissagK_Ijintro}
 Let $\pi: Z \rightarrow X$ be a closed immersion and assume that $2$ is invertible on both schemes. Then there is a canonical isomorphism
\begin{align}
H^i(Z, \I^j, \pi^! K^{ }) \simeq  H^i (X \on Z, \I^j, K^{ }).
\end{align}
\end{theorem}
Here $H^i (X, \I^j, K^{ })$ is the $i$-th hypercohomology of the sheafification of the complex~\eqref{eq:Rost-Schmid-complex-Wittintro}, where we replace everywhere Witt groups by the $j$-th power of the fundamental ideal $\I^j$, and $H^i (X \on Z, \I^j, K^{ })$ is the variant with support in $Z$.

In Sections~\ref{sec:tr} and~\ref{sec:twitr} we define the twisted residue and transfer maps on real schemes. The former is modeled over \cite[\S3]{jacobson}, and the latter arises naturally as a particular case of coherent duality. We use them in Section~\ref{sec:GerstenCXr} to define a precomplex of abelian groups of the form
\begin{align}
\label{eq:CVKintro}
\resizebox{\textwidth}{!}{$
\bigoplus\limits_{\mu_K(x)=m} C(x_r, \Z({K_{k(x)}})) 
\to
\bigoplus\limits_{\mu_K(x)=m+1} C(x_r, \Z({K_{k(x)}}))
\to\cdots\to
\bigoplus\limits_{\mu_K(x)=n} C(x_r, \Z({K_{k(x)}})).
$}
\end{align}
We show that the precomplex~\eqref{eq:CVKintro} is a complex in Theorem~\ref{theo:gersten-complex-real} by comparing it with the complex~\eqref{eq:Rost-Schmid-complex-Wittintro}. By sheafifying the complex~\eqref{eq:CVKintro} we obtain a complex of sheaves of abelian groups on $X_r$
\begin{align}
\label{eq:def_CXrKintro}
\bigoplus\limits_{\mu_K(x)=m} (i_{x_r})_* \Z({K_{k(x)}}) 
\to
\bigoplus\limits_{\mu_K(x)=m+1} (i_{x_r})_* \Z({K_{k(x)}})
\to \cdots \to
\bigoplus\limits_{\mu_K(x)=n} (i_{x_r})_* \Z({K_{k(x)}})
\end{align} 
which we denote by $\underline{G}(X_r,K^{ })$, called the \emph{Gersten complex}. 
Note that $\underline{G}(X_r, K)$ is a complex of acyclic sheaves, so its hypercohomology agree with the cohomology of the complex~\eqref{eq:CVKintro}. Furthermore we have the following result:
\begin{theorem}
If $2$ is invertible on $X$, then $H^i (X, \I^\infty, K^{ })$ agrees with the $i$-th hypercohomology group of $\underline{G}(X_r,K^{ })$.
\end{theorem}

In Sections~\ref{sec:realduality} and~\ref{sec:GWKcomplex} we justify the fundamental properties of the complex $\underline{G}(X_r, K)$. We start by building some refinements on the six functors on real schemes (cf.\  Theorems~\ref{th:poincare} and~\ref{thm:abspur}):
\begin{theorem}
\label{th:relabspur}
\begin{enumerate}
\item (Poincar\'e duality)
If $f:X\to Y$ is a smooth morphism of relative dimension $d$ with tangent bundle $T_f$, there is an 
isomorphism of functors
\begin{align}
f^*_r(-)\overset{}{\otimes}\mathbb{Z}(\operatorname{det}(T_f))[d]\to f^!_r(-).
\end{align}
\item (absolute purity)
Let $i:Z\to X$ be a closed immersion between regular schemes of pure codimension $c$, and denote by $N$ the normal bundle of $i$. Then for any locally constant constructible sheaf $\mathcal{F}$ on $X_r$ there is an 
isomorphism in $D(Z_r)$
\begin{align}
\label{eq:abspur}
i^*_r\mathcal{F}\overset{}{\otimes}\mathbb{Z}(\operatorname{det}(N)^{-1})[-c]\simeq i^!_r\mathcal{F}.
\end{align}
\end{enumerate}
\end{theorem}
For a more detailed discussion about duality and purity results in the motivic context, see \cite{DJK} and \cite{DFJK}. 
Now Theorem~\ref{th:relabspur} allows us to prove the main results of the paper (cf.\ Corollary~\ref{cor:sch2.3}, Theorem \ref{th:fupper!compat} and Theorem~\ref{th:CXKdual}):
\begin{theorem}\label{thm:main-theorem-D(X_r)}
\begin{enumerate}
\item
\label{num:sch2.3intro}
If $X$ is a regular excellent scheme with a dualizing complex $K$ and $ L$ is an invertible $\mathcal{O}_X$-module, then there is a canonical isomorphism $\mathbb{Z}( L)\simeq\underline{G}(X_r, L)$ in $D(X_r)$. 
\item 
\label{num:fupper!compatintro}
If $f:X\to Y$ is a quasi-projective morphism between excellent schemes and $K$ is a dualizing complex on $Y$, there is an isomorphism 
\begin{align}
\underline{G}(X_r,f^!K^{ })\simeq f_r^!\underline{G}(Y_r,K^{ })
\end{align}
which is compatible with compositions in the sense of~\eqref{eq:gammacomp}.
\item 
\label{num:th:CXKdualintro}
If $X$ is an excellent scheme with a dualizing complex $K$, then the complex $\underline{G}(X_r,K^{ })$ is a \emph{dualizing object} of $D(X_r)$ (see Definition~\ref{def:dualizing}).
\end{enumerate}
\end{theorem}
Theorem \ref{thm:main-theorem-D(X_r)} (1) is a generalization of \cite[Corollary 2.3]{Sch1}. The proof of Theorem \ref{thm:main-theorem-D(X_r)} (2) uses Theorem \ref{thm:main-theorem-D(X_r)} (1), Theorem~\ref{thm:devissagK_Ijintro} and Poincar\'e duality. We deduce Theorem \ref{thm:main-theorem-D(X_r)} (3) from Theorem \ref{thm:main-theorem-D(X_r)} (1) and (2) using some techniques with six functors. 

In Section~\ref{sec:realkun} we discuss some relations with the semialgebraic Borel-Moore homology over a real closed field (\cite[III \S2]{Del}). In particular, over the field of real numbers we recover the classical Borel-Moore homology in topology (\cite{BM}, see also \cite{Ver}, \cite[III]{Bre} and \cite[IX]{Iver}).

\section{Borel-Moore Devissage}
\label{sec:BMdev}

\subsection{The Gersten-Witt complex}
\label{sec:BWGcomplex}
\subsubsection{}
Throughout Section~\ref{sec:BWGcomplex}, we assume that $2$ is invertible in all  schemes. Our main references are \cite{BW} and \cite{Gil07}.
\subsubsection{}
Let $X$ be a scheme with a minimal dualizing complex $K$. We consider the triangulated category with duality 
\begin{align} 
(D^b_c(\Q(X)), \#^{K^{}}, \can^{K^{}}, \lambda^{K^{}})
\end{align}
in the sense of Schlichting \cite[Section 3]{Sch17}. For the translation of Balmer's $\delta$-exact duality and Schlichting's duality, we refer to  \cite[Section 2]{Xie19}. The codimension function $\mu_K: X \rightarrow \Z $ associated to $K^{ }$ induces a filtration on the derived category  $D^b_c(\Q(X))$ as
\begin{align} 
D^b_c(\Q(X)) = D^m \supseteq D^{m+1}  \supseteq \cdots \supseteq  D^n \supseteq (0)   
\end{align}
where $D^p:= \big\{ M \in  D^b_c(\Q(X)) : \mu_K (x) \geq  p \textnormal{ for all } x \in \support M   \big\} $.  Since $D^{p+1} \subseteq D^p$ is a saturated subcategory, we have an exact sequence of triangulated category with duality
\begin{align}
D^{p+1} \to D^p \xrightarrow{q} D^p/ D^{p+1}
\end{align}
which induces a long exact sequence of Witt groups
\begin{align} 
\cdots  \to 
\W^i(D^{p+1}, K^{ })
\to
\W^i(D^{p}, K^{ })    
\xrightarrow{q_*}
\W^i(D^p/ D^{p+1}, K^{ })
\xrightarrow{\partial} 
\W^{i+1}(D^{p+1}, K^{ }) 
\to\cdots.
\end{align}
Now, we get a complex
\begin{align} 
\W^{m}(D^m/D^{m+1}, K^{ }) 
\xrightarrow{\partial^q} 
\W^{m+1}(D^{m+1}/D^{m+2}, K^{ })
\xrightarrow{\partial^q}
\cdots 
\xrightarrow{\partial^q} 
\W^{n}(D^n, K^{ }).
\end{align}
The localization functor (which is an equivalence of triangulated categories)
\begin{align}
 D^p/ D^{p+1} 
 \longrightarrow 
 \coprod_{\mu_K(x)=p} D^b_{c,\mathfrak{m}_x}(\Q(\Oo_{X,x})) 
\end{align}
induces an isomorphism of Witt groups
\begin{align}
 \W^i(D^p/ D^{p+1}, K^{ }) 
 \simeq 
 \bigoplus\limits_{\mu_K(x)=p} \W^i_{\mathfrak{m}_x}(\Oo_{X,x}, K^{ }_x). 
\end{align}
For $x\in X$, denote by $\pi_x: \operatorname{Spec}(k(x))\rightarrow\operatorname{Spec}(\Oo_{X,x})$ the canonical map. By \cite[Lemma 4.3]{Xie20}, for any $i$ there is a devissage isomorphism 
\begin{align}
\label{eq:devisowitt}
 \pi_{x,*}:  \W^{i}(k(x), \pi_x^! K^{ }_x ) \simeq \W^i_{\mathfrak{m}_x}(\Oo_{X,x}, K^{ }_x) .
\end{align}
Putting all these together we get a complex 
\begin{align}
\label{eq:GWfirst}
\resizebox{\textwidth}{!}{$
\bigoplus\limits_{\mu_K(x)=m} \W^{m}(k(x), \pi_x^! K^{ }_x)
  \xrightarrow{\partial^q}
  \bigoplus\limits_{\mu_K(x)=m+1} \W^{m+1}(k(x), \pi_x^! K^{ }_x)   
  \to \cdots \to
  \bigoplus\limits_{\mu_K(x)=n} \W^{n}(k(x), \pi_x^! K^{ }_x),
$}
\end{align}
which can be rewritten as 
\begin{align} 
\label{eq:Gille-complex}
\bigoplus\limits_{\mu_K(x)=m} \W(k(x), K_{k(x)}) 
\xrightarrow{\partial^q} 
\bigoplus\limits_{\mu_K(x)=m+1} \W(k(x), K_{k(x)}) 
\to \cdots \to
\bigoplus\limits_{\mu_K(x)=n} \W(k(x), K_{k(x)})
\end{align}
by Section \ref{num:Kkx}. 

\subsection{The Rost-Schmid complex}\label{sect:Rost-Schmid-complex}
\subsubsection{}
Let $X$ be a scheme and let $K^{ }$ be a minimal dualizing complex on $X$. We construct the Rost--Schmid precomplex 
\begin{align} 
\label{eq:Rost-Schmid-complex-Witt}
\bigoplus\limits_{\mu_K(x)=m} \W(k(x), K_{k(x)}) 
\xrightarrow{\partial}
\bigoplus\limits_{\mu_K(x)=m+1} \W(k(x), K_{k(x)})
\to\cdots\to
\bigoplus\limits_{\mu_K(x)=n} \W(k(x), K_{k(x)})
\end{align}
in a different way from the complex~\eqref{eq:Gille-complex} by following the method of \cite{Schmid}. This complex shall be denoted by $C(X,W,K)$. Eventually in Theorem~\ref{theorem:Gille-Schmid} we show that the two constructions agree when $2$ is invertible on $X$. 

\subsubsection{}
 If $A$ is a one dimensional local domain with field of fractions $F$, we denote $\tilde{A}$ be its normalization. $\tilde{A}$ contains finitely many maximal ideals. Choose one maximal ideal $\tilde{\mathfrak{m}}$ of $\tilde{A}$ . Now, $\tilde{A}_{\tilde{\mathfrak{m}}}$ is a discrete valuation ring with field of fractions $F$ with residue field $k(\tilde{\mathfrak{m}})$, and $\tilde{A}_{\tilde{\mathfrak{m}}}$ is a finitely generated $A$ module. Thus, we have a finite morphism $\tilde{f}:\mathrm{Spec}(\tilde{A}_{\tilde{\mathfrak{m}}}) \rightarrow \mathrm{Spec}(A)$, and we form a commutative diagram
\begin{align}
\begin{split}
 \xymatrix{\mathrm{Spec}(\tilde{A}_{\tilde{\mathfrak{m}}}) \ar[r]^-{\tilde{f}} &  \mathrm{Spec}(A)   \\
	\mathrm{Spec}(k(\tilde{\mathfrak{m}} ) ) \ar[u]^-{\tilde{\pi}}  \ar[r]^g &  \mathrm{Spec}(k(\mathfrak{m})) \ar[u]^-{\pi} }
\end{split}
\end{align}
 
The Rost--Schmid residue map
\begin{align}
 \delta\colon \W(F) \rightarrow \W(\kmmt, (\mmt / \mmt^2)^* )
\end{align}
is defined by sending
a form $\langle a \rangle$ on $F$ to $ \delta_2^{\pi_{\tilde{y}}}(\langle a \rangle) \otimes \pi^*$
where $\delta_2$ is Milnor's second residue (cf. \cite[Lemma 1.2]{MH73}) and $\pi^*$ is the dual basis of the basis $\pi$ in the free rank one $\kmmt$ vector space  $\mmt / \mmt^2$. 

If $K^{ }$ is a minimal dualizing complex on $\mathrm{Spec}(A)$, then  $\tilde{f}^! K^{ }$ is a minimal dualizing complex on $\mathrm{Spec}(\tilde{A}_{\tilde{\mathfrak{m}}})$, cf. \cite[Sublemma 6.11]{Gil07}.\ Let $K_{\kmmt} : = (\tilde{f}^!K)_{\kmmt}$ and let $\Omega$ be the kernel of the first differential of $\tilde{f}^! K^{ }$(which is a dualizing module).  
We have a twisted residue map
\begin{align}
 	\delta: \W(F, K_F) \rightarrow \W(k(\tilde{\mathfrak{m}}),  K_{\kmmt} )      
\end{align}
induced by the Rost-schmid residue map, this is because we have isomorphisms $K_F \simeq \Omega_F$ and
\begin{align}
 K_{\kmmt} 
 \overset{\eqref{eq:lcipur}}{\simeq}
 \Omega \otimes_{\tilde{A}_{\mmt}} (\mmt/\mmt^2)^* 
 \simeq 
 \Omega_{\kmmt} \otimes_{\kmmt} (\mmt/\mmt^2)^*.
\end{align}
Note that $k(\mathfrak{m}) \subset k(\tilde{\mathfrak{m}})$ is a finite field extension, and there are canonical isomorphisms 
\begin{align}
K_{\kmmt} = (\tilde{f}^!K)_{\kmmt} \simeq g^! K_{\kmm}, 
\end{align}
and we define the twisted transfer as
\begin{align} 
\label{eq:trWitt}
tr\colon \W(\kmmt,  K_{\kmmt}   ) \simeq \W(\kmmt, g^! K_{\kmm}  ) \stackrel{g_*}\rightarrow \W(\kmm, K_{\kmm} ).
\end{align}

\begin{remark}
	The transfer map~\eqref{eq:trWitt} is well-defined even if $2$ is not invertible. Assume $F$ is any field, and $L$ is a finite field extension of $F$. Let $g: \Spec (L) \rightarrow \Spec (F)$ be the canonical map, and let $H$ be a rank one free module over $F$. Then, $g^!H : = \Hom_F(L,H)$ is a rank one free module over $L$. The transfer is defined to be $g_*: \W(L,g^!H) \rightarrow \W(F,H)$ by sending a symmetric space $\varphi : V \rightarrow \Hom_L(V, g^!H)$ to $ e \circ \varphi : V \rightarrow \Hom_F(V,H) $ where the map $e: \Hom_F(V_F,g^!H)  \rightarrow \Hom_F(V,H)$ is the ``evaluation at $1$'' map defined by sending $  f   $ to $ v \mapsto f(v) (1_L)$. 
\end{remark}

\subsubsection{}
\label{num:wittres}
Now we define the differential in the precomplex~\eqref{eq:Rost-Schmid-complex-Witt}
\begin{align}
 \partial\colon \bigoplus\limits_{\mu_K(x)=i} \W(k(x), K_{k(x)}) \rightarrow \bigoplus\limits_{\mu_K(y)=i+1} \W(k(y), K_{k(y)})
\end{align}
as the composition
\begin{align}
 \W(k(x), K_{k(x)}) \stackrel{\oplus \delta}\longrightarrow \bigoplus\limits_{ \tilde{y}} \W(k(\tilde{y}), K_{k(\tilde{y})}) \stackrel{\Sigma\textnormal{tr}}\longrightarrow  \W(k(y),K_{k(y)})  
\end{align}
where the sum is taken over all points $\tilde{y}$ that dominate y and
live inside the normalization $\widetilde{\overline{\{x\}}}$ of $\overline{\{x\}}$
in its residue field $k(x)$. Note that $\mathcal{O}_{\widetilde{\overline{\{x\}}}, \tilde{y}}$ is a discrete valuation ring dominating the one dimensional local domain $\mathcal{O}_{\overline{\{x\}}, y} $ in its field of fractions which is isomorphic to $k(x)$.

\begin{theorem}\label{theorem:Gille-Schmid} Assume that $2$ is invertible on $X$. 
	Then the Rost--Schmid precomplex \eqref{eq:Rost-Schmid-complex-Witt} coincides with the Gersten-Witt complex \eqref{eq:Gille-complex}. In particular, the precomplex \eqref{eq:Rost-Schmid-complex-Witt} is a complex.
\end{theorem}
\begin{proof}
 It is enough to check the one dimensional local domain case (cf.\ \cite[Proof of Lemma 7.2]{Gil07}). The proof in the local case can be extracted from \cite[Propostion 6.10]{Gil07}.
\end{proof}
  
\begin{remark}\label{rmk:gerstenconj}
\begin{enumerate}
\item In~\cite[\S3]{Plowman}, Plowman gave a construction similar to the precomplex \eqref{eq:Rost-Schmid-complex-Witt} via a Rost-Schmid type construction with residue complexes. Our construction seems to be comparable to his by translating dualizing complexes to residue complexes, but such a comparison is beyond the scope of this paper. Plowman has also informed us that he is able to show that his precomplex is a complex even if $2$ is not invertible.
\item The Gersten conjecture for Witt groups says that if $X$ is the spectrum of a regular local ring, the precomplex $C(X,W,K)$ in~\eqref{eq:Rost-Schmid-complex-Witt} is an acyclic complex. When $X$ has equal characteristic different from $2$, this conjecture is proved in \cite{BGPW}. In mixed characteristic, the conjecture is proved by Jacobson (\cite[Theorem 3.8]{jacobson2}) for any unramified regular local ring with $2$ invertible. 
For the two inverted Witt groups $\W(X)[{\frac{1}{2}}]$, this conjecture is proved by Jacobson (\cite[Theorem 5.3]{jacobson}) for any excellent regular local ring, and the excellence condition is later removed in \cite[Corollary E.4]{DFJK}. In fact, one can see that $C(X,W,K)[\frac{1}{2}] \simeq C(X_\QQ,W,K_{\QQ})[\frac{1}{2}]$, since Witt groups of a field with finite characteristic can only have $2$-primary torsions; therefore the Gersten conjecture for $C(X,W,K)[\frac{1}{2}]$ (trivially) follows from the characteristic $0$ case. 
\end{enumerate}
  \end{remark}

\subsection{$\I$-cohomology with dualizing complex} 
\label{sec:Icohdual}
\subsubsection{}
Throughout Section~\ref{sec:Icohdual}, we assume that $2$ is invertible in all schemes.
\begin{definition}
	If $k$ is a field and $L$ is a one-dimensional $k$-vector space, we define the \textbf{fundamental module} $\I^j(k,H) \subset \W(k,H)$ as follows: consider $\W(k,H)$ as a $\W(k)$-module, and $\I^j(k, H) $ is defined to be the submodule of $\I^j(k) \cdot \W(k,H) \subset \W(k,H)$ where $\I^j(k) \subset \W(k)$ is the $j$-th fundamental ideal. 
\end{definition}
\begin{lemma}\label{lem:Icohocommutewithdiff} For $x \in X$ with $\mu_I(x) = p$, and $y \in X$ with $\mu_I(y) = p+1$ such that $y$ lies in the closure of $x$. Then we have
\begin{align}
\partial^q_p(\I^j(k(x), K_{k(x)})) \subseteq \I^{j-1}(k(y), K_{k(y)} ).
\end{align}
\end{lemma}

\begin{proof} 
	It is enough to prove the clasim for $X$ the spectrum of a one dimensional local domain $A$ with dualizing complex $K$. Let $F$ be the field of fractions, and $\pi_\eta: \Spec (F) \rightarrow \Spec (A)$ the canonical map. In this case, we mush show that the map
	$ \partial^q: \W(F, K_\eta) \rightarrow  \W(k_\mathfrak{m}, K_\mathfrak{m} ) $ from the Gersten-Witt complex preserves fundamental modules. Choosing suitable injective hulls (i.e. trivializations), one has a residue map 
	$d_\iota: \W(F) \rightarrow \W(k_\mathfrak{m})$ and a cubical commutative diagram	
\begin{align}
\begin{split}
\xymatrix{ & \W(F, K_\eta)  \ar[dd]^(.3){\iota} \ar[rr]^-{\partial^q} &&  \W(k_\mathfrak{m}, K_\mathfrak{m} )\ar[dd]^-{\iota} \\
		\I^j(F, K_\eta)  \ar[ur]\ar[rr]^(.6){\partial^q}  \ar[dd]^-{\iota} && \I^{j-1}(k_\mathfrak{m},  K_\mathfrak{m} ) \ar[dd]^(.3){\iota} \ar[ur] \\
		& \W(F) \ar[rr]^(.3){d_\iota} && \W(k_\mathfrak{m})     \\
		\I^j(F)  \ar[ur]\ar[rr]^-{d_\iota} && \I^{j-1}(k_\mathfrak{m}) \ar[ur]     }
\end{split}
\end{align}
The bottom diagram is commutative by \cite[Theorem 6.6]{Gil07}. The back vertical diagram is commutative by \cite[Diagram (15)]{Gil07} and our definition $\partial^q$ via the devissage isomorphism~\eqref{eq:devisowitt}. The upper diagram is therefore commutative by the commutativity of other five diagrams.
\end{proof}

\begin{definition}
\label{def:CXIjK}
By Lemma~\ref{lem:Icohocommutewithdiff}, the Rost--Schmid complex \eqref{eq:Rost-Schmid-complex-Witt} restricts to fundamental modules as 
\begin{align}
\resizebox{\textwidth}{!}{$
\bigoplus\limits_{\mu_K(x)=m} \I^j(k(x), K_{k(x)}) 
\xrightarrow{\partial^q} 
\bigoplus\limits_{\mu_K(x)=m+1} \I^{j-1}(k(x), K_{k(x)})  
\to \cdots \to  
\bigoplus\limits_{\mu_K(x)=n} \I^{j-(n-m)}(k(x), K_{k(x)}).  
$}
\end{align}
Denote this complex by $G(X, \I^j, K^{ })$, which can be sheafified into a complex of Zariski sheaves of abelian groups $\underline{G}(X, \I^j, K^{ })$.
\end{definition}

\subsection{$\I$-cohomology with support}
\label{sec:Icohsupp}

\subsubsection{}
Throughout Section~\ref{sec:Icohsupp}, we assume that $2$ is invertible in all schemes.

\subsubsection{}
Let $\pi: Z \rightarrow X$ be a closed immersion, with $j: U \hookrightarrow X$ its open complement. From the complex $\underline{G}(X, \I^j, K^{ })$ in Definition~\ref{def:CXIjK}, we define the complexes of sheaves
\begin{align} 
\underline{G} (X \on Z, \I^j, K^{ }) : = \ker \big[ \underline{G} (X , \I^j, K^{ })  \stackrel{j^*} \longrightarrow \underline{G} (U, \I^j, K^{ })  \big] .
\end{align}

\begin{definition}
We define the $\I$-cohomology of the scheme $X$ with minimal dualizing complex $K$ supported on $Z$ to be the hypercohomology
\begin{align}
 H^i(X \on Z , \I^j, K^{ }) : = \mathbb{H}^{i}(\underline{G} (X \on Z, \I^j, K^{ }) ) 
\end{align}
We denote $H^i(X , \I^j, K)=H^i(X \on X , \I^j, K^{ })$.
\end{definition}

\subsubsection{}
Let $D^b_{c,Z}(\Q(X))$ be the full triangulated subcategory of $D^b_c(\Q(X))$ of complexes supported on $Z$, that is, complexes $\mathcal{F}$ such that $\support F\subseteq Z$.
Define 
\begin{align}
\D^p_{Z,X} : = \Big\{ M \in \D^b_{c,Z}(\mathcal{Q}(X))\mid \mu_K(x)\geq p  \textnormal{ for all $x \in \support(M)$}  \Big\}
\end{align}
considered as a full subcategory of $ \D^b_{c,Z}(\mathcal{Q}(X))$. Since the duality functor $ \#^{K^{ }}$ preserves the subcategory $\D^p_{Z,X} $, the pair $( \D^p_{Z,X},\#^{K^{ }})$ is a triangulated category with duality. The subcategories $\D^p_{Z,X} $ provide a finite filtration
\begin{align}
\D^b_{c,Z}(\mathcal{Q}(X)) =   \D^m_{Z,X} \supseteq \D_{Z,X}^{m+1} \supseteq \cdots \supseteq \D_{Z,X}^p \supseteq \cdots \supseteq \D_{Z,X}^n  \supseteq (0) 
\end{align}
which induces exact sequences of triangulated categories with duality
\begin{align}
 \D_{Z,X}^{p+1} \longrightarrow \D_{Z,X}^{p} \longrightarrow \D_{Z,X}^{p}/ \D_{Z,X}^{p+1}. 
\end{align}
On the other hand, we define
\begin{align}
\D_Z^{p}: = \Big\{ M \in \D^b_{c}(\mathcal{Q}(Z))\mid \mu_{\pi^!K}(x)\geq p  \textnormal{ for all $x \in \support(M)$}  \Big\}  
\end{align}
as a full subcategory of $ \D^b_{c}(\mathcal{Q}(Z))$. For the same reason, we have a finite filtration
\begin{align}
 \D^b_{c}(\mathcal{Q}(Z)) =   \D^m_Z \supseteq \D_Z^{m+1} \supseteq \cdots  \supseteq \D_{Z}^p \supseteq \cdots \supseteq \D_Z^n \supseteq (0) 
\end{align}
which induces exact sequences of triangulated categories with duality
\begin{align}
\D_{Z}^{p+1} \longrightarrow \D_{Z}^{p} \longrightarrow \D_{Z}^{p}/ \D_{Z}^{p+1}.   \end{align}
Since $\mu_{K} (z) = \mu_{ \pi^! K} (z)$ for all $z\in Z$, we have $\pi_*( \D_{Z}^{p}) \subset \D_{Z,X}^{p} $. It follows that we obtain a map of exact sequences of triangulated categories with duality
\begin{align}
    \begin{CD} (\D_{Z}^{p+1}, \#^{\pi^\flat K^{ }})  @>>> (\D_{Z}^{p}, \#^{\pi^\flat K^{ }}) @>>>  (\D_{Z,X}^{p}/ \D_{Z,X}^{p+1} , \#^{\pi^\flat K^{ }})  \\ 
@V{\pi_*}VV  @V{\pi_*}VV  @V{\pi_*}VV \\
(\D_{Z,X}^{p+1} , \#^{K^{ }}) @>>>  (\D_{Z,X}^{p}, \#^{K^{ }}) @>>> (\D_{Z,X}^{p}/ \D_{Z,X}^{p+1}, \#^{K^{ }})
\end{CD}     
\end{align}
which induces a map of long exact sequences of groups
\begin{align}
\label{map}  
\small{  \begin{CD} \cdots  \longrightarrow & \W^i  (\D_{Z}^{p}, \#^{\pi^\flat K^{ }})  & \longrightarrow & \W^i (\D_{Z}^{p}/ \D_{Z}^{p+1} , \#^{\pi^\flat K^{ }} )&\longrightarrow& \W^{i+1} (\D_{Z}^{p+1}, \#^{\pi^\flat K^{ }})   & \longrightarrow \cdots &\\ 
	& @V{\pi_*}VV  @V{\pi_*}VV  @V{\pi_*}VV  \\
	\cdots \longrightarrow & \W^i   (\D_{Z,X}^{p}, \#^{K^{ }}) &\longrightarrow& \W^i  (\D_{Z,X}^{p}/ \D_{Z,X}^{p+1}, \#^{K^{ }})  &\longrightarrow& \W^{i+1}  (\D_{Z,X}^{p+1} , \#^{K^{ }}) & \longrightarrow \cdots&\phantom{a} .
	\end{CD}   }  
\end{align}
On the other hand, the localization functors
\begin{align}
 loc: \D_{Z,X}^{p}/ \D_{Z,X}^{p+1} \rightarrow \prod_{x \in Z \cap X^{(p)}_{\mathcal{I}}} \D^b_{\mathfrak{m}_{X,x}}(\mathcal{O}_{X,x}) 
 \end{align}
are equivalences of categories. From this, we deduce a morphism of complexes
\begin{align}\label{devissagekey}  \small{  \begin{CD} \cdots   \longrightarrow & \bigoplus\limits_{x \in Z^{(p)}_{\pi^!K} } \W^p_{\mathfrak{m}_x}(\Oo_{Z,x}, \pi^\flat K^{ }_x) &\longrightarrow& \bigoplus\limits_{x \in Z^{(p+1)}_{\pi^!K} } \W^{p+1}_{\mathfrak{m}_x}(\Oo_{Z,x},\pi^\flat K^{ }_x) & \longrightarrow \cdots &\\ 
	&  @V{\pi_*}VV  @V{\pi_*}VV  \\
	\cdots \longrightarrow& \bigoplus\limits_{x \in Z \cap X^{(p)}_{K} } \W^p_{\mathfrak{m}_x}(\Oo_{X,x}, K^{ }_x)  &\longrightarrow& \bigoplus\limits_{x \in Z \cap X^{(p+1)}_{K} }  \W^{p+1}_{\mathfrak{m}_x}(\Oo_{X,x}, K^{ }_x)  & \longrightarrow \cdots&\phantom{a} .
	\end{CD}   }  
\end{align}
Now the devissage of Gille \cite{Gil07} (or \cite[Lemma 4.5]{Xie20}) says that all the vertical arrows in~\eqref{devissagekey} are isomorphisms, by reducing both groups to Witt groups of fields. The upper line of \eqref{devissagekey} computes $H^i(Z, \W, \pi^! K^{ })$ and the lower line computes $H^i (X \on Z, \W, K^{ })$. By Lemma \ref{lem:Icohocommutewithdiff}, we conclude that: 
\begin{theorem}
	\label{thm:devissagK_Ij}
 Let $\pi: Z \rightarrow X$ be a closed immersion and let $K^{ }$ be a minimal dualizing complex on $X$.  Assume that $2$ is invertible on both schemes. Then the map 
\begin{align}
 \pi_* : H^i(Z, \I^j, \pi^! K^{ }) \longrightarrow  H^i (X \on Z, \I^j, K^{ })
\end{align}
	induced by Diagram (\ref{devissagekey}) is an isomorphism for any $i, j \in \Z$. 
\end{theorem}

\section{The (twisted) signature map}
\label{sec:signature}

\subsection{Twisted residue} \label{sec:tr}
\subsubsection{}
Let $X$ be a scheme with a minimal dualizing complex $K^{ }$. Let $x,y \in X$ be such that $\mu_K(x)=p$ and $\mu_K(y)=p+1$.\ Recall that both $K_{k(x)}$ and $K_{k(y)}$ are one-dimensional vector spaces over their corresponding residue fields. Our goal in this section is to define a twisted differential
\begin{align}
\label{eq:partialre}
\partial_{re}\colon  C(x_r, \Z(K_{k(x)})) \rightarrow C(y_r, \Z(K_{k(y)})).
\end{align}
The map~\eqref{eq:partialre} will be defined by a twisted residue map followed by a twisted transfer map on the real spectrum.

\subsubsection{}
Let $R$ be a discrete valuation ring with field of fractions $F$ and let $\mathfrak{m}$ be its maximal ideal. Let $k_\mathfrak{m}$ be it is residue field. Choose a uniformizing parameter $\pi$ of $R$. Suppose $P$ is an ordering on $F$. We say $R$ is \textit{convex} on $(F,P)$ whenever for all $x,y,z \in F$, $x \leq_P z \leq_P y$ and  $x,y \in R$ implies $z \in R$. For any ordering $\bar{\xi}$ on $k_\mathfrak{m}$, the set
\begin{align}
Y_{\bar{\xi}}:= \{P \in \Sper (F) | R \textnormal{ is convex in } (F,P), \textnormal{ and $\bar{\xi} = \bar{P}$ on $k_\mathfrak{m}$}  \}
\end{align}
maps bijectively to $\{\pm 1\}$ by sending $P$ to $\sign_P(\pi)$
\cite[Section 3]{jacobson}. Denote $\xi_\pm^\pi$ the ordering on $F$ such that $\sign_{\xi_\pm^\pi}(\pi) = \pm1$. Jacobson \cite[Section 3]{jacobson} defines a group homomorphism
\begin{align}
\label{eq:betares}
\begin{split}
\beta_\pi\colon C(\Sper(F), \Z) &\rightarrow C(\Sper(k_\mathfrak{m}), \Z)\\
s &\mapsto \left(\bar{\xi}\mapsto s(\xi_+^\pi) - s(\xi_-^\pi)\right)
\end{split}
\end{align}
called the residue map. The map~\eqref{eq:betares} depends on the choice of uniformizing parameter. 
\begin{remark}
	Note that by~\ref{sec:2invert}, if $\kmm$ is finite characteristic, 
	then the target of the $\beta_\pi$ is zero. 
\end{remark}

\subsubsection{}
For a field $Q$ and a one-dimensional $Q$-vector space $H$, 
by Definition~\ref{def:realtw}, the group $C(\Sper(Q), \Z(H))$ is the group
\begin{align}
C(\Sper(Q), \Z) \otimes_{\Z[Q^\times]} \Z[H^\times]
\end{align}
where the underlying map
\begin{align}
\zeta\colon \Z[Q^\times] \rightarrow C(\Sper(Q), \Z)
\end{align}
is the ring homomorphism given by
\begin{align}
\left(\textstyle\sum n_a a\right) \mapsto \left( \xi \mapsto \sum n_a \sign_\xi a\right).
\end{align}
	The twisted residue map is the group homomorphism
\begin{align}
\begin{split}
	\beta' \colon C(\Sper(F), \Z) &\rightarrow C(\Sper(k_{\mathfrak m}), \Z({(\mathfrak{m}/\mathfrak{m}^2)^*}))\\
	s &\mapsto \beta_\pi(s) \otimes \pi^*
\end{split}
\end{align}
	where $\pi^*$ is the dual basis of the basis $\pi$ of the one-dimensional $k_{\mathfrak{m}}$-vector space  $\mathfrak{m}/\mathfrak{m}^2$.\

\begin{lemma}
	The twisted residue map $\beta'$ does not depend on the choice of the uniformizing parameter.
\end{lemma}
\begin{proof}
	Let $a \pi$ be another uniformizing parameter with $a \in R - \mathfrak{m}$.
	Note that $\beta'_{a\pi}(s)(\bar{\xi})  = \sign_P(a)\beta'_\pi(s)(\bar{\xi})$ and $\sign_P(a) = \sign_{\bar{\xi}}(a)$.
	It follows that
\begin{align}
	\beta'_{a\pi}(s)  \otimes {a\pi}^* = \beta'_{a\pi}(s) \zeta(a)  \otimes {\pi}^*
\end{align}
	Now, $\beta'_{a\pi}(s) \zeta(a)(\bar{\xi})= \sign_{\bar{\xi}}(a)^2\beta'_\pi(s)(\bar{\xi})   =\beta'_\pi(s)(\bar{\xi}).$
	The result follows.
\end{proof}

\subsubsection{}
Now let $H$ be a free $R$-module of rank $1$, and let $t:R\rightarrow H$ be a trivialization. Then $t_F: F\rightarrow H_F$ and $t_{\km}: \km \rightarrow H_{\km} $ are both trivializations, which induce isomorphisms
\begin{align}
t_F: C(\Sper(F), \Z)\simeq C(\Sper(F), \Z(H_F) )
\end{align}
\begin{align}
t_{\km}: C(\Sper(\km), \Z) \simeq C(\Sper(\km), \Z(H_{\km}))
\end{align}
\begin{definition}\label{Def:twist-residue}
The \textbf{twisted residue map} is the group homomorphism
\begin{align}
\label{eq:twistres}
\begin{split}
\beta\colon C(\Sper(F), \Z(H_F)) &\rightarrow C(\Sper(k_{\mathfrak m}), \Z({ H_{\km} \otimes (\mathfrak{m}/\mathfrak{m}^2)^*}))\\
s &\mapsto t_{\km} \beta(t_F^{-1}(s))
\end{split}
\end{align}
which does not depend on the choice of trivialization.
\end{definition}
\subsection{Twisted transfer}
\label{sec:twitr}
\subsubsection{}
Let $F\to L$ be a finite field extension and let $P$ be an ordering on $F$. We say that an ordering $R$ on $L$ is an extension of $P$ if the image of $P$-positive elements in $L$ are $R$-positive. 
Define a map
\begin{align}
\label{eq:CFtr}
\begin{split}
t\colon C(\Sper(L), \Z) &\rightarrow C(\Sper(F), \Z)\\
\phi &\mapsto \left(P \mapsto \sum_{R\supset P} \phi(R)\right)
\end{split}
\end{align}
where the sum runs through all extensions $R$ of $P$.  The map~\eqref{eq:CFtr} is well-defined as the sum is finite, and will be called the transfer map.\footnote{Note that the number of such extensions equals $\sign_P(Tr_*\langle 1 \rangle)$, cf.\ \cite[\S3, Theorem~4.5]{Scha}).}
\newcommand{\tr}{\mathrm{tr}}

Denote by $f: \Spec (L) \rightarrow \Spec (F)$ the canonical map. If $H$ is a one-dimensional $F$-vector space, then we have $f^! H : = \Hom_F(L, H)$. There is a trace form $\tr : L \rightarrow \Hom_F(L,F)$ which gives rise to an isomorphism $\tr_H : H_L \rightarrow \Hom_F(L,H)  = f^!H$ by tensoring by $H$ on both sides.  
\begin{definition}\label{Def:twisted-transfer}
	The \textbf{twisted transfer map} is the group homomorphism
\begin{align}
\label{eq:twistedtr}
\begin{split}
t\colon C(\Sper(L), \Z(f^! H)) &\rightarrow C(\Sper(F), \Z(H))\\
\phi &\mapsto \left(P \mapsto \sum_{R\supset P} \tr_H^{-1}\phi(R)\right)
\end{split}
\end{align}
given by the composition
\begin{align}
C(\Sper(L), \Z(f^! H)) 
\xrightarrow{\id \otimes \tr_H^{-1}}
C(\Sper(L), \Z(H_L))=C(\Sper(L), \Z(H))  
\xrightarrow{t\otimes \id} 
C(\Sper(F), \Z(H)).
\end{align}
\end{definition}

\subsection{The Gersten complex of a minimal dualizing complex}
\label{sec:GerstenCXr}
\subsubsection{}
\label{num:betadiff}
  If $A$ is a one dimensional local domain with field of fractions $F$, we denote $\tilde{A}$ be its normalization. $\tilde{A}$ contains finitely many maximal ideals. Choose one maximal ideal $\tilde{\mathfrak{m}}$ of $\tilde{A}$ . Now, $\tilde{A}_{\tilde{\mathfrak{m}}}$ is a discrete valuation ring with field of fractions $F$ with residue field $k(\tilde{\mathfrak{m}})$, and $\tilde{A}_{\tilde{\mathfrak{m}}}$ is a finitely generated $A$ module. We have a map $\tilde{f}:\mathrm{Spec}(\tilde{A}_{\tilde{\mathfrak{m}}}) \rightarrow \mathrm{Spec}(A)$. If $K^{ }$ is a dualizing complex on $A$, then  $\tilde{f}^! K^{ }$ is a dualizing complex on $\mathrm{Spec}(\tilde{A}_{\tilde{\mathfrak{m}}})$. We have a commutative diagram
\begin{align}
\begin{split}
 \xymatrix{\mathrm{Spec}(\tilde{A}_{\tilde{\mathfrak{m}}}) \ar[r]^-{\tilde{f}} &  \mathrm{Spec}(A)   \\
	\mathrm{Spec}(k(\tilde{\mathfrak{m}} ) ) \ar[u]^-{\tilde{\pi}}  \ar[r]^g &  \mathrm{Spec}(k(\mathfrak{m})) \ar[u]^-{\pi} }
\end{split}
\end{align}

Let $K_{\kmmt} : = (\tilde{f}^!K)_{\kmmt}$ and let $\Omega$ be the kernel of the first differential.\ We have a residue map
\begin{align}
\label{eq:beta}
	\beta: C(\Sper(F), \Z(K_F)) \rightarrow C(\Sper(k(\tilde{\mathfrak{m}})), \Z( K_{\kmmt} ) )\end{align}
induced by $\beta'$ above via Definition \ref{Def:twist-residue}, since we have $K_F \simeq \Omega_F$ and
\begin{align}
K_{\kmmt} 
\overset{\eqref{eq:lcipur}}{\simeq}
\Omega \otimes_{\tilde{A}_{\mmt}} (\mmt/\mmt^2)^* 
\simeq
\Omega_{\kmmt} \otimes_{\kmmt} (\mmt/\mmt^2)^*.
\end{align}

Note that $k(\mathfrak{m}) \subset k(\tilde{\mathfrak{m}})$ with $\kmmt$ a finite extension of $k(\mathfrak{m})$, so we have $K_{\kmmt}  = g^! K_{\kmm}$. By Definition \ref{Def:twisted-transfer} the twisted transfer~\eqref{eq:twistedtr} can be expressed as
\begin{align}
\label{eq:tsp}
	t\colon C(\Sper(\kmmt), \Z( K_{\kmmt}   ))  \simeq C(\Sper(\kmmt), \Z(g^! K_{\kmm}  ) ) \rightarrow C(\Sper(k(\mathfrak{m})), \Z(K_{k(\mathfrak{m})   }  ) ).
\end{align}

\subsubsection{}
\label{num:Cres}
Let $X$ be a scheme with a minimal dualizing complex $K^{ }$. If $x$ and $y$ are two points of $X$, by the same procedure as in~\ref{num:wittres}, we define the differential mentioned in~\eqref{eq:partialre}
\begin{align}
\label{eq:partre}
 \partial_{re}\colon \bigoplus\limits_{\mu_K(x)=i} C(x_r, \Z({K_{k(x)}})) \rightarrow \bigoplus\limits_{\mu_K(y)=i+1} C(y_r, \Z({K_{k(y)}}))
\end{align}
as the composition
\begin{align}
 C(x_r, \Z({K_{k(x)}})) \xrightarrow{\eqref{eq:beta}} \bigoplus\limits_{ \tilde{y}} 
 C(\tilde{y}_r, \Z({K_{k(\tilde{y})}}))
  \xrightarrow{\eqref{eq:tsp}} C(y_r, \Z({K_{k(y)}}))
\end{align}
where the sum is taken over all points $\tilde{y}$ that dominate y and
live inside the normalization $\widetilde{\overline{\{x\}}}$ of $\overline{\{x\}}$
in its residue field $k(x)$. 

\begin{theorem}\label{theo:gersten-complex-real}
Let $X$ be a scheme with minimal dualizing complex $K$. 
Then the following precomplex, denoted as $G(X_r, K^{ })$, is a complex of abelian groups:
\begin{align}
\label{eq:CXrK}
\resizebox{\textwidth}{!}{$
\bigoplus\limits_{\mu_K(x)=m} C(x_r, \Z({K_{k(x)}})) 
\xrightarrow{\partial_{re}}
\bigoplus\limits_{\mu_K(x)=m+1} C(x_r, \Z({K_{k(x)}}))
\to\cdots\to
\bigoplus\limits_{\mu_K(x)=n} C(x_r, \Z({K_{k(x)}}))
$}
\end{align}
\end{theorem}
\begin{proof} By~\ref{sec:2invert} we may assume that $X$ has characteristic $0$, therefore with $2$ invertible in $X$. We construct a commutative ladder diagram 
	\begin{align*}
	\label{eq:def_CXrK}
	\resizebox{\textwidth}{!}{$
	\xymatrix{ 
		\bigoplus\limits_{\mu_K(x)=m} \W(k(x), {K_{k(x)}}) \ar[r]^-{\partial} \ar[d]^-{\sign}&  \bigoplus\limits_{\mu_K(x)=m+1} \W(k(x), {K_{k(x)}}) \ar[r] \ar[d]^-{2 \cdot\sign} & \cdots \ar[r]  & \bigoplus\limits_{\mu_K(x)=n} \W(k(x), {K_{k(x)}}) \ar[d]^-{2^{n-m}\cdot \sign} \\
		\bigoplus\limits_{\mu_K(x)=m} C(x_r, \Z({K_{k(x)}})) \ar[r]^-{\partial_{re}} &  \bigoplus\limits_{\mu_K(x)=m+1} C(x_r, \Z({K_{k(x)}})) \ar[r] & \cdots \ar[r]  & \bigoplus\limits_{\mu_K(x)=n} C(x_r, \Z({K_{k(x)}}))    }
		$}
	\end{align*} 
	We shall pick up one square in the ladder diagram to check the commutativity
\begin{align}
\begin{split}
 \xymatrix{ \W(k(x), {K_{k(x)}}) \ar[r]^-{\delta} \ar[d]^-{\sign}&  \bigoplus\limits_{\tilde{y}} \W(k(\tilde{y}), {K_{k(\tilde{y})}})  \ar[d]^-{2 \cdot\sign} \ar[r]^-{\tr} & \W(k(y), {K_{k(y)}})  \ar[d]^-{2 \cdot\sign}  \\
	 C(x_r, \Z({K_{k(x)}})) \ar[r]^-{\beta} &  \bigoplus\limits_{\tilde{y}} C(\tilde{y}_r, \Z({K_{k(\tilde{y})}})) \ar[r]^-{t} &   C(y_r, \Z({K_{k(y)}})) }
\end{split}
\end{align}
  where the sum is taken over all points $\tilde{y}$ that dominate y and
 live inside the normalization $\widetilde{\overline{\{x\}}}$ of $\overline{\{x\}}$
 in its residue field $k(x)$. The problem is reduced to one dimensional case. To explain this, we shall denote $A : = (\mathcal{O}_{X,x})_y$ which is a one dimensional local domain, and therefore $F: =k(x)$ is the field of fraction of $A$ and $\kmm : = k(y)$ is the residue field of $A$. If $\tilde{A}$ is the normalization of $A$ inside $F$, then $\tilde{A}_{\kmmt}$ is a discrete valuation ring with residue field $\kmmt : = k(\tilde{y})$. The commutativity of the left square is obtained by \cite[Lemma 3.1 and 3.2]{jacobson} and \cite[Lemma A.8]{HWXZ}. The commutativity of the right square is obtained by a similar argument in \cite[Lemma A.9]{HWXZ}.
 
 Recall that $\I^\infty(k(x), K_{k(x)})$ is the colimit of the sequence
\begin{align} 
\I^0(k(x), K_{k(x)}) \xrightarrow{\langle \langle -1\rangle \rangle} \I^1(k(x), K_{k(x)}) \xrightarrow{\langle \langle -1\rangle \rangle} \cdots \rightarrow \I^j(k(x), K_{k(x)}) \rightarrow \cdots  
\end{align}
 One can promote the ladder diagram of Witt groups $\W(k(x), K_{k(x)})$ to the group $\I^\infty(k(x), K_{k(x)})$, and we obtain an isomorphism of complexes (as in \cite[Corollary A.11 and Corollary A.12]{HWXZ}

\resizebox{\textwidth}{!}{$$
 \xymatrix{ 
 	\bigoplus\limits_{\mu_K(x)=m} \I^\infty (k(x), {K_{k(x)}}) \ar[r]^-{\partial} \ar[d]^-{\sign}&  \bigoplus\limits_{\mu_K(x)=m+1} \I^\infty (k(x), {K_{k(x)}}) \ar[r] \ar[d]^-{\sign} & \cdots \ar[r]  & \bigoplus\limits_{\mu_K(x)=n} \I^\infty(k(x), {K_{k(x)}}) \ar[d]^-{\sign} \\
 	\bigoplus\limits_{\mu_K(x)=m} C(x_r, \Z({K_{k(x)}})) \ar[r]^-{\partial_{re}} &  \bigoplus\limits_{\mu_K(x)=m+1} C(x_r, \Z({K_{k(x)}})) \ar[r] & \cdots \ar[r]  & \bigoplus\limits_{\mu_K(x)=n} C(x_r, \Z({K_{k(x)}}))    }
$$}
The vertical maps are isomorphisms by a theorem of Arason-Knebusch (\cite[Proposition 2.7]{jacobson}). The upper line is a complex by Theorem \ref{theorem:Gille-Schmid} and Lemma \ref{lem:Icohocommutewithdiff}. We conclude that the lower line is also a complex. 
\end{proof}

\begin{definition}
Let $X$ be a scheme with a minimal dualizing complex $K$. 
For any open constructible subset $U$ of $X_r$, we denote by $G(U, K^{ })$ the following subcomplex of $G(X_r, K^{ })$ in~\eqref{eq:CXrK}
\begin{align}
\label{eq:CVK}
\resizebox{\textwidth}{!}{$
\bigoplus\limits_{\mu_K(x)=m} C(x_r\cap U, \Z({K_{k(x)}})) 
\xrightarrow{\partial_{re}}
\bigoplus\limits_{\mu_K(x)=m+1} C(x_r\cap U, \Z({K_{k(x)}}))
\xrightarrow{\partial_{re}}
\cdots
\xrightarrow{\partial_{re}}
\bigoplus\limits_{\mu_K(x)=n} C(x_r\cap U, \Z({K_{k(x)}})).
$}
\end{align}
By sheafifying the complex $U\mapsto G(U, K^{ })$ in~\eqref{eq:CVK}, we obtain a complex of sheaves of abelian groups on $X_r$
		\begin{align}
	\label{eq:def_CXrK}
\bigoplus\limits_{\mu_K(x)=m} (i_{x_r})_* \Z({K_{k(x)}}) 
\xrightarrow{\partial_{re}}
\bigoplus\limits_{\mu_K(x)=m+1} (i_{x_r})_* \Z({K_{k(x)}})
\to \cdots \to
\bigoplus\limits_{\mu_K(x)=n} (i_{x_r})_* \Z({K_{k(x)}})
	\end{align} 
	which we denote by $\underline{G}(X_r,K^{ })$. In what follows we call it the \textbf{Gersten complex} on $X_r$.
\end{definition}

\subsubsection{}
If $K$ and $K'$ are two minimal dualizing complexes on $X$ which are quasi-isomorphic in $D^b_c(\mathcal{Q}(X))$, then there is an isomorphism  $G(X_r, K^{ })\simeq G(X_r, K')$ in $D(X_r)$. Since any dualizing complex is quasi-isomorphic to a mininal dualizing complex (cf. \cite{Gil07}), we can extend the definition of the complexes $G(U, K^{ })$ and $\underline{G}(X_r,K^{ })$ to any dualizing complex $K$, which are well-defined objects in $D(Ab)$ and $D(X_r)$ respectively.

\subsubsection{}
\label{num:Gerlb}
Note that since the twisting operation is compatible with tensor products of invertible modules (\ref{eq:twimorel}), if $L$ is an invertible $\mathcal{O}_X$-module, we have 
\begin{align}
\label{eq:Gtwilb}
\underline{G}(X_r, K^{ }\otimes L[n])=\underline{G}(X_r, K^{ })\otimes \mathbb{Z}(L)[n].
\end{align}

\subsubsection{}
Recall that we have defined in Definition~\ref{def:CXIjK} a bounded complex of Zariski sheaves $\underline{G}(X, \I^j, K^{ })$ on $X$. We denote
\begin{align}
\underline{G}(X,\I^\infty,K^{ })=\underset{j}{\operatorname{colim}}\ \underline{G}(X,\I^j,K^{ }).
\end{align} 

\begin{proposition}\label{Prop: I-cohomology-Real} Let $X$ be a scheme with a minimal dualizing complex $K$. Assume that $2$ is invertible in $X$. Then there are isomorphisms between hypercohomology groups
\begin{align}
H^i(X , \I^\infty, K)
\simeq
\mathbb{H}^{i}(\underline{G}(X,\I^\infty,K^{ })) 
\simeq
\mathbb{H}^{i}(\underline{G}(X_r , K^{ })).
\end{align} 
\end{proposition}
\begin{proof}
The first isomorphism is established in \ref{theorem:Gille-Schmid}. By the proof of Theorem \ref{theo:gersten-complex-real}, there is an isomorphism of complexes $\underline{G}(X,\I^\infty, K) \simeq \support_* \underline{G}(X_r, \Z, K)$. Since the functor $\support_*$ is exact (\cite[Theorem 19.2]{Sch1}), the second isomorphism follows from \cite[Lemma 4.6]{jacobson} by replacing the Grothendieck spectral sequence in \textit{loc. cit.} with the hypercohomology spectral sequence.
\end{proof}

\section{Scheiderer's theorems revisited}
\label{sec:realduality}
The goal of this section is to prove some generalizations of the results in \cite{Sch1}. 
\subsection{Six functors and purity theorems}
\subsubsection{}
We first recall the six functors on real schemes. For any scheme $X$, we denote by $D(X_r)$ the derived category of sheaves of abelian groups on the real scheme $X_r$. By \cite[Theorem 35]{Bac},  there is an equivalence of categories
\begin{align}
\label{eq:reteq}
D(X_r)\simeq D_{\mathbb{A}^1}(X)[\rho^{-1}]
\end{align}
between the derived category of $X_r$ and the $\rho$-inverted $\mathbb{A}^1$-derived category of $X$; in addition, the association $X\mapsto D(X_r)$ defines a \emph{motivic triangulated category} in the sense of \cite[Definition 2.4.45]{CD}. By \cite[Theorem 2.4.50]{CD}, we have the corresponding six functors formalism on the fibered category $X\mapsto D(X_r)$, and the equivalence~\eqref{eq:reteq} is compatible with the six functors. In particular, if $f:X\to Y$ is a separated morphism of schemes of finite type, there is a pair of adjoint functors given by the (derived) exceptional direct image and inverse image functors
\begin{align}
Rf_{r!}:D(X_r)\rightleftharpoons D(Y_r):f^!_r.
\end{align}
These functors are also have more explicit descriptions: for the induced map $f_r:X_r\to Y_r$ between the underlying real schemes, the functor $Rf_{r!}$ is the derived functor of the direct image with compact support functor on sheaves
\begin{align}
f_{r!}:Sh(X_r)\to Sh(Y_r)
\end{align}
(\cite[II.8]{Del}, \cite[3.1]{EP}). If $f=p\circ j$ is a compactification of $f$, that is, $p$ is a proper morphism and $j$ is an open immersion (\cite{Con}), then $f_{r!}=p_{r*}\circ j_{r!}$, where $j_{r!}$ is the extension by zero functor. 

\subsubsection{}
Recall the following general definition of Thom spaces in motivic categories (\cite[Example 2.4.3.1]{CD}):
\begin{definition}
Let $\mathcal{T}$ be a motivic triangulated category. Let $p:V\to X$ be the projection of a vector bundle over a scheme $X$, with $s:X\to V$ the zero section. The ($\mathcal{T}$-valued) \textbf{Thom space} $Th(V)\in\mathcal{T}(X)$ is defined as $Th(V)=Rp_!Rs_*Ls^*p^!\mathbbold{1}_X$, where $\mathbbold{1}_X\in\mathcal{T}(X)$ is the unit object.
\end{definition}
As a consequence of the six functors formalism (\cite[Theorem 2.4.50]{CD}), we have the following properties:
\begin{lemma}
\begin{enumerate}
\item (Localization)
If $i:Z\to X$ is a closed immersion with open complement $j:U\to X$, then there is a canonical distinguished triangle of endofunctors of $D(X_r)$:
\begin{align}
Ri_{r*}i^!_r\xrightarrow{}1\xrightarrow{}Rj_{r*}j^!_r\xrightarrow{+1}Ri_{r*}i^!_r[1].
\end{align}
\item (Projection formula)
There is a canonical isomorphism
\begin{align}
R\underline{Hom}(Rf_{r!}\mathcal{F},\mathcal{G})
\simeq 
Rf_{r*}R\underline{Hom}(\mathcal{F},f_r^!\mathcal{G}).
\end{align}
\item (Relative purity)
If $f:X\to Y$ is a smooth morphism with tangent bundle $T_f$, then 
there is an isomorphism of functors
\begin{align}
\label{eq:relpurity}
Th(T_f)\overset{}{\otimes}f^*_r(-)\simeq f^!_r(-).
\end{align}
\end{enumerate}
\end{lemma}

\subsubsection{}
For a vector bundle $V$ over a scheme $X$, the cohomology of the Thom space computes the cohomology groups of $V$ with support in $X$. 
More precisely, for any separated morphism of finite type $p:Y\to X$, we denote by $M_X(Y)=Rp_!p^!\mathbb{Z}\in D(X_r)$ the motive of $Y$ over $X$. Then by \cite[2.4.13]{CD}, the Thom space fits into a distinguished triangle
\begin{align}
Th(V)[-1]\to M_X(V^\times)\to M_X(V)\xrightarrow{+1}Th(V)
\end{align}
where $V^\times$ is the complement in $V$ of the zero section. In the case where $V$ is the trivial bundle of rank $1$, the global section $-1$ of $V^\times=\mathbb{G}_{m,X}$ gives a map
\begin{align}
\label{eq:sec-1}
\mathbb{Z}\xrightarrow{-(-1)_*} M_X(\mathbb{G}_{m,X})
\end{align}
which factors through $Th(\Oo_X)[-1]$ and induces a map
\begin{align}
\label{eq:Thomstr}
\rho:\mathbb{Z}\to Th(\Oo_X)[-1].
\end{align}
The map~\eqref{eq:Thomstr} is an isomorphism, since the category $D(X_r)$ is $\rho$-stable by \cite[Theorem 35]{Bac}.
%
More generally, the Thom space of the trivial bundle can be computed as
\begin{align}
\label{eq:Thomtriv}
Th(\Oo^d_X)
\simeq
Th(\Oo_X)^{\oplus d}
\overset{\eqref{eq:Thomstr}}{\simeq}
\mathbb{Z}[d].
\end{align}
More generally, we have the following result:
\begin{lemma}
\label{lm:Thdet}
Let $X$ be a scheme and let $V$ be a vector bundle of rank $d$ over $X$. Then in $D(X_r)$ there is an isomorphism
\begin{align}
\label{eq:Thdet}
Th(V)\simeq\mathbb{Z}(\operatorname{det}(V))[d]
\end{align}
which is compatible with pullbacks and short exact sequences of vector bundles.
\end{lemma}
\proof
We first deal with the case where $V=L$ is a line bundle.
Since both sides are Zariski sheaves, we will construct an isomorphism
\begin{align}
\label{eq:ThLB}
Th(L)\simeq\mathbb{Z}(L)[1]
\end{align}
by gluing local isomorphisms. Let $X=\bigcup_i U_i$ be an open cover of $X$ such that the restriction of $L$ to each $U_i$ is trivial. Then for each $i$ there is an isomorphism
\begin{align}
\phi_i:Th(L)_{|U_i}\simeq Th(\Oo_{U_i})\overset{\eqref{eq:Thomstr}}{\simeq}\mathbb{Z}[1]\simeq\mathbb{Z}(L_{|U_i})[1].
\end{align}
To show that the $\phi_i$'s glue into a global isomorphism, it suffices to show that $\phi_{i|U_i\cap U_j}=\phi_{j|U_i\cap U_j}$. 
Let $U=U_i\cap U_j$. By considering the transition maps attached to the local trivializations, it suffices to show that for any $u\in\mathcal{O}^\times(U)$, the following diagram commutes
\begin{align}
\label{eq:diagThtriv}
\begin{split}
  \xymatrix{
    Th(\Oo_{U})[-1]  \ar[d]_-{\times u} & \mathbb{Z} \ar[r]^-{\sim} \ar[l]_-{\sim}^-{\eqref{eq:Thomstr}} \ar[d]^-{v} & \mathbb{Z}\otimes_{\Z[\support^*\Oo_{U}^\times]} \mathbb{Z}[\support^*(L_U)^\times] \ar[d]^-{u}\\
    Th(\Oo_{U})[-1]  & \mathbb{Z} \ar[r]^-{\sim} \ar[l]_-{\sim}^-{\eqref{eq:Thomstr}} & \mathbb{Z}\otimes_{\mathbb{Z}[\support^*\Oo_{U}^\times]}\mathbb{Z}[\support^*(L_U)^\times].
  }
\end{split}
\end{align}
To see this, consider the middle vertical map $v\in \operatorname{End}(\mathbb{Z})=C(U_r,\mathbb{Z})$ given by sending $(x,P)$, where $x\in U$ and $P$ is an ordering on the residue field of $x$, to $\sign_P(u_x)\in\{\pm1\}$. We need to show that both squares of the diagram~\eqref{eq:diagThtriv} commute with this choice of $v$. The one on the right commutes by construction, and the one on the left reduces to the commutativity of the following diagram
\begin{align}
\label{eq:diagThtriv}
\begin{split}
  \xymatrix{
    M_U(\mathbb{G}_{m,U}) \ar[d]_-{\times u} & \mathbb{Z} \ar[l]_-{\eqref{eq:sec-1}}^-{} \ar[d]^-{v} \\
    M_U(\mathbb{G}_{m,U}) & \mathbb{Z} \ar[l]_-{\eqref{eq:sec-1}}^-{}
  }
\end{split}
\end{align}
which, by looking at stalks, follows from the topological fact that the multiplicative group of a real closed field has exactly two connected components.

In the general case, assume that $V$ is a vector bundle of rank $d$ over $X$. We know that $H^0(X_r,\mathbb{Z})=H^0(X_{ret},\mathbb{Z})$ is a Zariski sheaf, so by \cite[Theorem 5.3]{Ana} the category $D(X_r)$ has a canonical $SL$-orientation. By~\eqref{eq:Thomtriv}, for $V$ a vector bundle of rank $r$, we have an isomorphism
\begin{align}
\label{eq:Thdet2}
Th(V)
\simeq 
Th(\operatorname{det}(V))\otimes Th(\Oo^{d-1}_X)
\overset{\eqref{eq:Thomtriv}}{\simeq}
Th(\operatorname{det}(V))[d-1]
\simeq
\overset{\eqref{eq:ThLB}}{\simeq}
\mathbb{Z}(\operatorname{det}(V))[d]
\end{align}
where the first isomorphism is induced by the weak Thom class associated to the $SL$-orientation.

From the gluing procedure above, we know that the isomorphism~\eqref{eq:ThLB} is compatible with pullbacks and tensor products of line bundles. Since weak Thom class are compatible with pullbacks and short exact sequences of vector bundles, and if $0\to V_1\to V_2\to V_3\to0$ is a short exact sequence of vector bundles we have a canonical isomorphism $\operatorname{det}(V_2)\simeq\operatorname{det}(V_1)\otimes\operatorname{det}(V_3)$, we know that the isomorphism~\eqref{eq:Thdet2} is compatible with pullbacks and short exact sequences of vector bundles.
\endproof

\begin{remark}
\begin{enumerate}
\item
In the case where $X$ is regular, one can construct the isomorphism~\eqref{eq:Thdet} directly by combining \cite[Theorem 2.48]{HWXZ} and the isomorphism~\eqref{eq:jac8.6intro}.
The nature of this argument is the existence of a Gysin morphism for the zero section of $L$ (see \cite[Proposition 4.3.10 1)]{DJK}). Unfortunately for $X$ singular we do not know how to construct such a Gysin morphism from the Witt-theoretic side, as all the constructions in the literature are built on coherent Witt groups instead of the usual ones.
\item
The same arguments as in Lemma~\ref{lm:Thdet} gives an isomorphism $Th(V)\simeq\mathbb{Z}(\operatorname{det}(V)^{-1})[d]$. This reflects the fact that for any line bundle $L$ there is a canonical isomorphism $Th(L)\simeq Th(L^{-1})$, see \cite[Lemma 4.1]{Ana}, \cite[Proposition 2.2]{Ron}.
\end{enumerate}
\end{remark}

\subsubsection{}
From Lemma~\ref{lm:Thdet} and relative purity~\eqref{eq:relpurity} we deduce the following result:
\begin{theorem}[Poincar\'e duality]
\label{th:poincare}
If $f:X\to Y$ is a smooth morphism of relative dimension $d$ with tangent bundle $T_f$, there is an isomorphism of functors
\begin{align}
\label{eq:poincare}
f^*_r(-)\otimes\mathbb{Z}(\operatorname{det}(T_f))[d]\to f^!_r(-).
\end{align}

\end{theorem}

\subsubsection{}
The isomorphism~\eqref{eq:poincare} is compatible with compositions in the following sense: if $V\xrightarrow{g}X\xrightarrow{f}Y$ are composable smooth morphisms of relative dimensions $d_g$ and $d_f$ respectively, then the composition
\begin{align}
\begin{split}
&g^*_rf^*_r(-)\otimes\mathbb{Z}(\operatorname{det}(T_{f\circ g}))[d_f+d_g]\\
\simeq
&g^*_r(f^*_r(-)\otimes\mathbb{Z}(\operatorname{det}(T_f))[d_f])\otimes\mathbb{Z}(\operatorname{det}(T_g))[d_g]\\
\simeq
&g^!_r(f^*_r(-)\otimes\mathbb{Z}(\operatorname{det}(T_f))[d_f])
\simeq 
g^!_rf^!_r(-)
\end{split}
\end{align}
agrees with the isomorphism~\eqref{eq:poincare} for the composition $f\circ g$. Indeed, this follows from the short exact sequence of vector bundles over $V$ 
\begin{align}
0\to g^{-1}T_f\to T_{f\circ g}\to T_g\to 0
\end{align}
(\cite[Proposition 17.2.3]{EGA4}), Lemma~\ref{lm:Thdet}, and the fact that the purity isomorphism~\eqref{eq:relpurity} is compatible with compositions (\cite[1.5]{Ayo}).

Similarly, the isomorphism~\eqref{eq:poincare} is compatible with base change in the following sense:  for any Cartesian square
\begin{align}
\begin{split}
  \xymatrix@=10pt{
    W\ar[r]^-{q} \ar[d]_-{g} & X \ar[d]^-{f} \\
    Z \ar[r]^-{p} & Y
  }
  \end{split}
\end{align}
with $f$ smooth of relative dimension $d$, then the following diagram is commutative:
\begin{align}
\begin{split}
  \xymatrix@=10pt{
    g_r^!p_r^!(-)\ar[rr]^-{\sim} \ar@{=}[d]_-{} & & g_r^*p_r^!(-)\otimes\mathbb{Z}(\operatorname{det}(T_g))[d] \ar[d]^-{\wr} & \\
    q_r^!f_r^!(-) \ar[r]^-{\sim} & q_r^!(p_r^*(-)\otimes\mathbb{Z}(\operatorname{det}(T_f))[d]) & q_r^!p_r^*(-)\otimes\mathbb{Z}(\operatorname{det}(T_g))[d]. \ar[l]_-{\sim}
  }
  \end{split}
\end{align}

\subsubsection{}
Let $i:Z\to X$ be a closed immersion between regular schemes of pure codimension $c$, and denote by $N$ the normal bundle of $i$. We know that $i$ is in particular a regular closed immersion, and by the theory of \emph{fundamental classes} (\cite[4.3.1]{DJK}), there is a natural transformation of functors
\begin{align}
\label{eq:purtrans}
i_r^*(-) \overset{}{\otimes} Th(N)^{\otimes(-1)}\to i_r^!(-)
\end{align}
called \emph{purity transformation}, which is constructed using deformation to the normal cone. By Lemma~\ref{lm:Thdet} we obtain, for every object $\mathcal{F}\in D(X_r)$, a natural transformation in $D(X_r)$
\begin{align}
\label{eq:abspur0}
i_r^*(\mathcal{F}) 
\otimes
\mathbb{Z}(\operatorname{det}(N)^{-1})[-c]
\overset{\eqref{eq:Thdet}}{\simeq}
i_r^*(\mathcal{F}) \overset{}{\otimes} Th(N)^{\otimes(-1)}
\xrightarrow{\eqref{eq:purtrans}}
i_r^!(\mathcal{F}).
\end{align}
Recall that a sheaf of abelian groups $\mathcal{F}$ on $X_r$ is \textbf{constructible} if there is a finite stratification of $X_r$ into locally closed constructible subsets $M_i$ such that $\mathcal{F}_{|M_i}$ is the constant sheaf associated to a finitely generated abelian group (\cite[Definition A.3]{Sch}). 
The following result is a refinement of \cite[Theorem 1.7]{Sch1}:
\begin{theorem}[Absolute purity]
\label{thm:abspur}
If $\mathcal{F}$ is a locally constant constructible sheaf on $X_r$, then the map~\eqref{eq:abspur0} is an isomorphism. In other words, there is an isomorphism in $D(Z_r)$
\begin{align}
\label{eq:abspur}
i^*_r\mathcal{F}\overset{}{\otimes}\mathbb{Z}(\operatorname{det}(N)^{-1})[-c]\simeq i^!_r\mathcal{F}.
\end{align}
\end{theorem}
\proof
Since $\mathcal{F}$ is a locally constant constructible sheaf, then $\mathcal{F}$ is a \emph{dualizable object} in $D(X_r)$, and by \cite[Proposition 5.4]{FHM} there is a canonical isomorphism 
\begin{align}
i^*_r\mathcal{F}\overset{}{\otimes}i^!_r\mathbb{Z}\simeq i^!_r\mathcal{F}
\end{align}
and it suffices to show that the map~\eqref{eq:abspur0} is an isomorphism when $\mathcal{F}=\mathbb{Z}$ is the constant sheaf. 

By~\ref{sec:2invert}, we may assume that all schemes have characteristic $0$. Then we follow the arguments of \cite[Appendix C]{DFJK}. If both $X$ and $Z$ are smooth over a base field, then the result follows from Poincar\'e duality (Theorem~\ref{th:poincare}). In the general case, since the problem is Zariski local, we may also assume that both $X$ and $Z$ are affine. By Popescu's theorem (\cite[Theorem 1.1]{Spi}), the pair of regular schemes $(X,Z)$ is a projective limit of closed pairs of smooth $\mathbb{Q}$-schemes with transverse affine transition morphisms, and the result follows from the smooth case (see \cite[Remark 4.3.12 (ii)]{DJK} for more details).
\endproof



\begin{remark}

Theorem~\ref{thm:abspur} refines Scheiderer's absolute purity theorem (\cite[Theorem 1.7]{Sch1}) by giving a precise global identification of the object $i^!_r\mathcal{F}$ (instead of merely locally isomorphic to $i^*_r\mathcal{F}[-c]$), and also by removing the restriction to excellent schemes.

\end{remark}

\subsection{The Cousin complex}
\subsubsection{}
In 
this section we study the \emph{Cousin complex}. The following lemma is a variant of \cite[Corollary 2.2]{Sch1}, based on the style of \cite[4.2]{BO}:
\begin{lemma}
\label{lm:sch22}
Let $X$ be an excellent regular scheme, let $U$ be an open constructible subset of $X_r$ and let $\mathcal{F}$ be a locally constant constructible sheaf on $U$. Then there is a complex of abelian groups called the \textbf{Cousin complex}, denoted as $Cous(U,\mathcal{F})$
\begin{align}
\label{eq:Sch2.2}
\bigoplus_{x\in X^{(0)}}H^0(x_r\cap U,\mathcal{F}(\omega_{x/X}))
\to \bigoplus_{x\in X^{(1)}}H^0(x_r\cap U,\mathcal{F}(\omega_{x/X}))
\to\cdots \to \bigoplus_{x\in X^{(n)}}H^0(x_r\cap U,\mathcal{F}(\omega_{x/X}))
\end{align}
natural in $U$ and $\mathcal{F}$, whose $q$-th coohmology group is canonically isomorphic to $H^q(U,\mathcal{F})$ for all $q\geqslant0$.
\end{lemma}
\proof
The proof is very similar to \cite[Corollary 2.2]{Sch1}: we have the coniveau spectral sequence
\begin{align}
\label{eq:coniveauss}
E^{p,q}_1=\bigoplus_{x\in X^{(p)}\cap U}H^{p+q}_{x_r}(U,\mathcal{F})\Longrightarrow H^{p+q}(U,\mathcal{F}).
\end{align}
By absolute purity (Theorem~\ref{thm:abspur}), for any $x\in X^{(p)}\cap U$ there is a canonical isomorphism $H^{p}_{x_r}(U,\mathcal{F})\simeq H^{0}_{}(x_r\cap U,\mathcal{F}(\omega_{x/X}))$, and $H^{n}_{x_r}(U,\mathcal{F})=0$ if $n\neq p$.
Consequently the spectral sequence~\eqref{eq:coniveauss} degenerates at $E_2$, and its $E_1$-page is concentrated on the line $q=0$, which gives rise to the complex~\eqref{eq:Sch2.2}.
\endproof

\subsubsection{}
\label{num:cousmap}
In particular, Lemma~\ref{lm:sch22} defines a map of complexes
\begin{align}
\label{eq:cousmap}
\mathcal{F}(U)\to Cous(U,\mathcal{F})
\end{align}
where the map $\mathcal{F}(U)\to \bigoplus_{x\in X^{(0)}}H^0(x_r\cap U,\mathcal{F}(\omega_{x/X}))$
can be identified with
\begin{align}
\label{eq:cousmapdet}
\mathcal{F}(U)
\to 
\underset{V\subset X \textrm{open}}{\operatorname{lim}}\mathcal{F}(V_r\cap U)
\simeq
\bigoplus_{x\in X^{(0)}}H^0(x_r\cap U,\mathcal{F})
=
\bigoplus_{x\in X^{(0)}}H^0(x_r\cap U,\mathcal{F}(\omega_{x/X})).
\end{align}
Here every point $x\in X^{(0)}$ is a generic point of $X$, and therefore there is a canonical trivialization of $\omega_{x/X}$.

\subsubsection{}
By Lemma~\ref{lm:sch22}, sheafifying the map~\eqref{eq:cousmap} gives the following result:
\begin{corollary}
Let $X$ be an excellent regular scheme and let $\mathcal{F}$ be a locally constant constructible sheaf on $X_r$. Then there is a resolution of $\mathcal{F}$ by acyclic sheaves:
\begin{align}
\bigoplus_{x\in X^{(0)}}(i_{x_r})_*\mathcal{F}(\omega_{x/X})
\to \bigoplus_{x\in X^{(1)}}(i_{x_r})_*\mathcal{F}(\omega_{x/X})
\to \cdots \to \bigoplus_{x\in X^{(n)}}(i_{x_r})_*\mathcal{F}(\omega_{x/X})
\end{align}
\end{corollary}

\subsubsection{}
Let $X$ be an excellent regular scheme and let $U$ be an open constructible subset of $X_r$. Let $ L$ be an invertible $\mathcal{O}_X$-module, which can be resolved by a minimal dualizing complex $K_L$ (\ref{num:regcod}). In the case where $\mathcal{F}=\mathbb{Z}(L)$, we have the following result:
\begin{lemma}
\label{lm:cousG}
For $\mathcal{F}=\mathbb{Z}(L)$, the complex $Cous(U,\mathbb{Z}(L))$ in~\eqref{eq:Sch2.2} agrees with the complex $G(U, K_L)$ in~\eqref{eq:CVK}.
\end{lemma}
\proof
We may assume $U=X_r$. For any point $x$ of $X$, 
by~\eqref{eq:lcipur} we have an isomorphism 
\begin{align}
 L_x\otimes\omega_{x/X} \simeq  L_{k(x)},
\end{align}
which induces an isomorphism $\mathbb{Z}( L)(\omega_{x/X})\simeq\mathbb{Z}( L_{k(x)})$. To check differential maps agree in each degree, we only need to check the first differential maps agree, which follows our definition of differentials in~\ref{num:betadiff} and \cite[Lemma 4.9]{jacobson} (see also \cite[Proposition 2.6]{Sch1}), which finishes the proof.
\endproof
By sheafifying we obtain the following result, generalizing \cite[Corollary 2.3]{Sch1}:
\begin{corollary}
\label{cor:sch2.3}
Let $X$ be an excellent regular scheme and let $ L$ be an invertible $\mathcal{O}_X$-module. Then the Gersten complex $\underline{G}(X_r, K_L)$ in~\eqref{eq:def_CXrK} is a resolution of the sheaf $\mathbb{Z}( L)$ by acyclic sheaves.
\end{corollary}

\newpage

\section{Functoriality of the Gersten complex}
\label{sec:GWKcomplex}




The goal of this section is to prove the following result:
\begin{theorem}
\label{th:fupper!compat}
Let $f:X\to Y$ be a quasi-projective morphism of excellent schemes and let $K^{ }$ be a dualizing complex on $Y$. Then there is an isomorphism in $D(X_r)$
\begin{align}
\label{eq:isofupper!}
\gamma_f:\underline{G}(X_r,f^!K^{ })\simeq f_r^!\underline{G}(Y_r,K^{ })
\end{align}
which is compatible with compositions. 
\end{theorem}
In other words, the Gersten complex $\underline{G}(X_r,K^{ })$ in \eqref{eq:def_CXrK} commutes with the functor $f^!$.

\subsection{Closed immersions}

\subsubsection{}
Let $i:Z\to X$ be an immersion of schemes and let $K$ be a minimal dualizing complex on $X$. Then $i^!K$ is a minimal dualizing complex on $Z$. For $x$ a point of $Z$, we consider $i(x)$ as a point of $X$. Then we have a canonical identification of $1$-dimensional $k(x)$-vector spaces
\begin{align}
\label{eq:iK}
K_{k(i(x))}=(i^!K)_{k(x)}
\end{align}
and $\mu_K(i(x))=\mu_{i^!K}(x)$.

\subsubsection{}
Let $i:Z\to X$ be a closed immersion with open complement $j:X-Z\to X$. Any open constructible subset $U$ of $X_r$ can be written as a disjoint union
\begin{align}
U=(Z_r\cap U)\coprod((X-Z)_r\cap U)
\end{align}
and by~\eqref{eq:iK} there is a canonical (split) short exact sequence of complexes
\begin{align}
\label{eq:locG}
0\to
G(Z_r\cap U, i^!K^{ })
\to
G(U, K^{ })
\to
G((X-Z)_r\cap U, j^!K^{ })
\to
0.
\end{align}
The short exact sequence~\eqref{eq:locG} is functorial with respect $U$, and therefore by sheafifying we obtain a short exact sequence of sheaves
\begin{align}
\label{eq:locGsh}
0\to
i_*\underline{G}(Z_r,i^!K^{ })
\to
\underline{G}(X_r,K^{ })
\to
j_*\underline{G}((X-Z)_r,j^!K^{ })
\to
0.
\end{align}
Since $j$ is an open immersion, we have $\underline{G}((X-Z)_r,K^{ })=j^*\underline{G}(X_r,K^{ })$. In the derived category $D(X_r)$, the sequence~\eqref{eq:locGsh} identifies the object $i_*\underline{G}(Z_r,i^!K^{ })$ with the fiber of the map 
\begin{align}
\underline{G}(X_r,K^{ })\to j_*j^*\underline{G}(X_r,K^{}).
\end{align}
By the gluing formalism (see \cite[IV Proposition 14.6]{SGA4}, \cite[Proposition 2.3.3]{CD}), we have the localizing fiber sequence
\begin{align}
i_*i^!\to 1\to j_*j^*
\end{align}
from which we deduce an isomorphism in $D(X_r)$
\begin{align}
\label{eq:Galpha}
\alpha_i:
\underline{G}(Z_r,i^!K^{ })\simeq i^!\underline{G}(X_r,K^{ }).
\end{align}
Note that the isomorphism~\eqref{eq:Galpha} is a version of Theorem~\ref{thm:devissagK_Ij} for real schemes.

\subsection{Smooth morphisms}

\subsubsection{}
Let $f:X\to Y$ be a smooth morphism of excellent schemes of relative dimension $d$ and let $K$ be a dualizing complex on $Y$. Then we have an isomorphism in $D(X_r)$
\begin{align}
\label{eq:sm*!}
\underline{G}(X,f^*K)\otimes\mathbb{Z}(\operatorname{det}(T_f))[d]
\overset{\eqref{eq:Gtwilb}}{\simeq}
\underline{G}(X,f^*K\otimes\operatorname{det}(T_f)[d])
\overset{\eqref{eq:lcipur}}{\simeq}
\underline{G}(X,f^!K).
\end{align}

\subsubsection{}
For every point $y\in Y$, consider the Cartesian square
\begin{align}
\begin{split}
  \xymatrix@=10pt{
    X_y \ar[r]^-{\tau_y} \ar[d]_-{f_y}  & X \ar[d]^-{f} \\ 
    y \ar[r]^-{i_y} & Y.
  }
\end{split}
\end{align}
Let $U$ be an open constructible subet of $X_r$. Then there is a map
\begin{align}
\label{eq:CGsm}
\begin{split}
&C(y_r\cap f_r(U), \Z({K_{k(y)}})) 
=
C(f_r((X_y)_r\cap U), \Z({K_{k(y)}}))\\
\xrightarrow{}
&C((X_y)_r\cap U, f_r^*\Z({K_{k(y)}}))
=
C((X_y)_r\cap U, \Z({(f_y)^*K_{k(y)}}))\\
\xrightarrow{\eqref{eq:cousmap}}
&Cous((X_y)_r\cap U,\Z({(f_y)^*K_{k(y)}}))
=
G((X_y)_r\cap U,(f_y)^*K_{k(y)})
\end{split}
\end{align}
where the last equality follows from Lemma~\ref{lm:cousG}. Sheafifying the map~\eqref{eq:CGsm} gives a map of complexes
\begin{align}
\label{eq:frGsh}
f_r^*(i_{y_r})_*\mathbb{Z}(K_{k(y)})
\to
(\tau_{y_r})_*\underline{G}({X_y}_r,(f_y)^*K_{k(y)}).
\end{align}
\begin{lemma}
\label{lm:shsmpb}
The map~\eqref{eq:frGsh} is an isomorphism in $D(X_r)$.
\end{lemma}
\proof
We have the following isomorphisms:
\begin{align}
\label{eq:frGshiso}
f_r^*(i_{y_r})_*\mathbb{Z}(K_{k(y)})
\simeq 
(\tau_{y_r})_*f^*_{y_r}\mathbb{Z}(K_{k(y)})
=
(\tau_{y_r})_*\mathbb{Z}((f_y)^*K_{k(y)})
\simeq
(\tau_{y_r})_*\underline{G}({X_y}_r,(f_y)^*K_{k(y)}).
\end{align}
where the first isomorphism is the smooth base change (\cite[16.11]{Sch}), and the last isomorphism follows from Corollary~\ref{cor:sch2.3}. It is straightforward that maps~\eqref{eq:frGsh} and~\eqref{eq:frGshiso} agree, which finishes the proof.
\endproof

\subsubsection{}
Note that $(f_y)^*K_{k(y)}$ is a line bundle on the regular scheme $X_y$, and therefore the codimension function on $X_y$ defined by $(f_y)^*K_{k(y)}$ agrees with that of the structure sheaf (\ref{num:codlb} and \ref{num:regcod}). By \cite[III 7.3 and 8.7 5)]{Har}, for every point $x$ of $X_y$ we have
\begin{align}
\label{eq:codfib}
\mu_{f^*K}(x)=\mu_{\mathcal{O}_{X_y}}(x)+\mu_K(y)
\end{align}
and
\begin{align}
\label{eq:kxfib}
((f_y)^*K_{k(y)})_{k(x)}\simeq(f^*K)_{k(x)}.
\end{align}

\subsubsection{}
Let $U$ be an open constructible subet of $X_r$, and let $m$ and $n$ be the minimum and maximum of the function $\mu_K$. We define a double complex of abelian groups denoted as $D(U,f,K)$, as follows:
\begin{align}
\begin{split}
\resizebox{\textwidth}{!}{$
  \xymatrix{
    \underset{\mu_K(y)=m}{\bigoplus}\underset{x\in X_{y}^{(0)}}{\bigoplus}C(x_r\cap U, \Z({(f^*K)_{k(x)}})) \ar[r]^-{\partial_{re}} \ar[d]_-{\partial_{re}}  & \underset{\mu_K(y)=m+1}{\bigoplus}\underset{x\in X_{y}^{(0)}}{\bigoplus}C(x_r\cap U, \Z({(f^*K)_{k(x)}})) \ar[r]^-{\partial_{re}} \ar[d]_-{\partial_{re}} & \cdots \ar[r]^-{\partial_{re}} & \underset{\mu_K(y)=n}{\bigoplus}\underset{x\in X_{y}^{(0)}}{\bigoplus}C(x_r\cap U, \Z({(f^*K)_{k(x)}})) \ar[d]_-{\partial_{re}} \\ 
    \underset{\mu_K(y)=m}{\bigoplus}\underset{x\in X_{y}^{(1)}}{\bigoplus}C(x_r\cap U, \Z({(f^*K)_{k(x)}})) \ar[r]^-{-\partial_{re}} \ar[d]_-{\partial_{re}} & \underset{\mu_K(y)=m+1}{\bigoplus}\underset{x\in X_{y}^{(1)}}{\bigoplus}C(x_r\cap U, \Z({(f^*K)_{k(x)}})) \ar[r]^-{-\partial_{re}} \ar[d]_-{\partial_{re}} & \cdots \ar[r]^-{-\partial_{re}} & \underset{\mu_K(y)=n}{\bigoplus}\underset{x\in X_{y}^{(1)}}{\bigoplus}C(x_r\cap U, \Z({(f^*K)_{k(x)}})) \ar[d]_-{\partial_{re}} \\
    \vdots \ar[d]_-{\partial_{re}} & \vdots \ar[d]_-{\partial_{re}} & \ddots & \vdots \ar[d]_-{\partial_{re}} \\
    \underset{\mu_K(y)=m}{\bigoplus}\underset{x\in X_{y}^{(d)}}{\bigoplus}C(x_r\cap U, \Z({(f^*K)_{k(x)}})) \ar[r]^-{(-1)^d\partial_{re}} & \underset{\mu_K(y)=m+1}{\bigoplus}\underset{x\in X_{y}^{(d)}}{\bigoplus}C(x_r\cap U, \Z({(f^*K)_{k(x)}})) \ar[r]^-{(-1)^d\partial_{re}} & \cdots \ar[r]^-{(-1)^d\partial_{re}} & \underset{\mu_K(y)=n}{\bigoplus}\underset{x\in X_{y}^{(d)}}{\bigoplus}C(x_r\cap U, \Z({(f^*K)_{k(x)}}))
  }
$}
\end{split}
\end{align}
where
\begin{itemize}
\item The entries are given by direct sums of groups of the form $C(x_r\cap U, \Z({(f^*K)_{k(x)}}))$, for points $y\in Y$ and $x\in X_y$, arranged into a matrix according to the intergers $\mu_K(y)$ and $\mu_{\mathcal{O}_{X_y}}(x)$, the codimension of $x$ in the fiber $X_y$.
\item The vertical maps are given by the map $\partial_{re}$ in~\eqref{eq:partre}, and the horizontal maps are given by the map $(-1)^i\partial_{re}$ whenever $x$ has codimension $i$ in $X_y$, which are well-defined by~\eqref{eq:codfib}. 
\end{itemize}
This is a well-defined complex since $\partial_{re}\circ\partial_{re}=0$ by Theorem~\ref{theo:gersten-complex-real}. In addition, by~\eqref{eq:kxfib}, each column of the complex $D(U,f,K)$ is a direct sum of complexes of the form $G((X_y)_r\cap U,(f_y)^*K_{k(y)})$ (placed vertically), and the complex $D(U,f,K)$ can be written in the forllowing form:
\begin{align}
\resizebox{\textwidth}{!}{$
\underset{\mu_K(y)=m}{\bigoplus}G((X_y)_r\cap U,(f_y)^*K_{k(y)})
\xrightarrow{}
\underset{\mu_K(y)=m+1}{\bigoplus}G((X_y)_r\cap U,(f_y)^*K_{k(y)})
\xrightarrow{}
\cdots
\xrightarrow{}
\underset{\mu_K(y)=n}{\bigoplus}G((X_y)_r\cap U,(f_y)^*K_{k(y)}).
$}
\end{align}
We then have a map of double complexes
\begin{align}
\label{eq:GtoD}
G(f_r(U),K)\to D(U,f,K)
\end{align}
given by the following diagram
\begin{align}
\begin{split}
\resizebox{\textwidth}{!}{$
  \xymatrix{
    \underset{\mu_K(y)=m}{\bigoplus}C(y_r\cap f_r(U), \Z({K_{k(y)}})) \ar[r]^-{\partial_{re}} \ar[d]_-{\eqref{eq:CGsm}}  & \underset{\mu_K(y)=m+1}{\bigoplus}C(y_r\cap f_r(U), \Z({K_{k(y)}})) \ar[r]^-{\partial_{re}} \ar[d]_-{\eqref{eq:CGsm}} & \cdots \ar[r]^-{\partial_{re}} & \underset{\mu_K(y)=n}{\bigoplus}C(y_r\cap f_r(U), \Z({K_{k(y)}})) \ar[d]_-{\eqref{eq:CGsm}} \\ 
    \underset{\mu_K(y)=m}{\bigoplus}G((X_y)_r\cap U,(f_y)^*K_{k(y)}) \ar[r]^-{} & \underset{\mu_K(y)=m+1}{\bigoplus}G((X_y)_r\cap U,(f_y)^*K_{k(y)}) \ar[r]^-{} & \cdots \ar[r]^-{} & \underset{\mu_K(y)=n}{\bigoplus}G((X_y)_r\cap U,(f_y)^*K_{k(y)})
  }
$}
\end{split}
\end{align}
which is commutative by~\ref{num:cousmap}. By~\eqref{eq:codfib}, the total complex of the complex $D(U,f,K)$ is nothing but the complex $G(U,f^*K)$, and therefore the map~\eqref{eq:GtoD} gives rise to a map of complexes
\begin{align}
\label{eq:GUfupper}
G(f_r(U),K)\to \operatorname{Tot}(D(U,f,K))=G(U,f^*K).
\end{align}
By Lemma~\ref{lm:shsmpb}, sheafifying the map~\eqref{eq:GUfupper} gives an isomorphism in $D(X_r)$
\begin{align}
\label{eq:Gsmpb}
f_r^*\underline{G}(Y,K)\simeq \underline{G}(X,f^*K).
\end{align}
We deduce from the map~\eqref{eq:Gsmpb} an isomorphism in $D(X_r)$
\begin{align}
\label{eq:Gbeta}
\begin{split}
\beta_f:
&f_r^!\underline{G}(Y,K)
\overset{\eqref{eq:poincare}}{\simeq}
f_r^*\underline{G}(Y,K)\otimes\mathbb{Z}(\operatorname{det}(T_f))[d]\\
\overset{\eqref{eq:Gsmpb}}{\simeq}
&\underline{G}(X,f^*K)\otimes\mathbb{Z}(\operatorname{det}(T_f))[d]
\overset{\eqref{eq:sm*!}}{\simeq}
\underline{G}(X,f^!K).
\end{split}
\end{align}

\subsection{General case}

\subsubsection{}
\label{num:compE}
Let $S$ be a scheme and $\operatorname{QProj}_S$ be the category of quasi-projective $S$-schemes. We suppose given, for every scheme $X$ in $\operatorname{QProj}_S$, an object $E_X\in D(X_r)$. Then given two composable morphisms $X\xrightarrow{f}Y\xrightarrow{g}Z$ in $\operatorname{QProj}_S$, and two maps $\alpha:E_X\to f_r^!E_Y$, $\beta:E_Y\to g_r^!E_Z$, we denote by $\beta\cdot\alpha$ the composition
\begin{align}
\beta\cdot\alpha:E_X\xrightarrow{\alpha}f_r^!E_Y\xrightarrow{\beta}f_r^!g_r^!E_Z\simeq (g\circ f)_r^!E_Z.
\end{align}
The following lemma is a standard result in intersection theory (see \cite{FM}, \cite{DJK}):
\begin{lemma}
\label{lm:gamma}
In the setting of~\ref{num:compE}, assume given the following additional data:
\begin{enumerate}
\item
For every closed immersion $i:Z\to X$ in $\operatorname{QProj}_S$, a map 
\begin{align}
\label{eq:alphaci}
\alpha_i:E_Z\to i_r^!E_X\in D(Z_r);
\end{align}
\item
For every smooth morphism $f:Y\to X$ in $\operatorname{QProj}_S$, a map 
\begin{align}
\label{eq:betasm}
\beta_f:E_Y\to f_r^!E_X\in D(Y_r);
\end{align}
\end{enumerate}
satisfying the following relations:
\begin{enumerate}
\item
For two composable closed immersions $W\xrightarrow{k}Z\xrightarrow{i}X$ in $\operatorname{QProj}_S$, we have 
\begin{align}
\label{eq:cicomp}
\alpha_i\cdot\alpha_{k}=\alpha_{i\circ k}:E_W\to k_r^!i_r^!E_X.
\end{align}
In other words, the formation of the maps $\alpha_i$ in~\eqref{eq:alphaci} is compatible with compositions.

\item
For two composable smooth morphisms $V\xrightarrow{g}X\xrightarrow{f}Y$ in $\operatorname{QProj}_S$, we have 
\begin{align}
\label{eq:smcomp}
\beta_f\cdot\beta_{g}=\beta_{f\circ g}:E_V\to g_r^!f_r^!E_Y.
\end{align}
In other words, the formation of the maps $\beta_f$ in~\eqref{eq:betasm} is compatible with compositions.


\item
For every smooth morphism $f:X\to Y$ in $\operatorname{QProj}_S$ with a section $i:Y\to X$, 
we have
\begin{align}
\label{eq:seccmp}
\beta_f\cdot\alpha_i= id:E_Y\to E_Y. 
\end{align}

\item
For every Cartesian square of schemes in $\operatorname{QProj}_S$
\begin{align}
\begin{split}
  \xymatrix@=10pt{
    W\ar[r]^-{k} \ar[d]_-{g} & X \ar[d]^-{f} \\
    Z \ar[r]^-{i} & Y
  }
  \end{split}
\end{align}
where $i$ is a closed immersion and $f$ is smooth, 
we have 
\begin{align}
\label{eq:cscm}
\beta_f\cdot\alpha_{k}= \alpha_i\cdot\beta_{g}:E_W\to g_r^!i_r^!E_X.
\end{align}

\end{enumerate}
Then there is a unique way to associate to every morphism $f:X\to Y$ in $\operatorname{QProj}_S$ a map 
\begin{align}
\label{eq:gammaqp}
\gamma_f:E_X\to f_r^!E_Y\in D(X_r)
\end{align}
such that
\begin{enumerate}
\item
If $f$ is a closed immersion, then $\gamma_f=\alpha_f$.
\item
If $f$ is smooth, then $\gamma_f=\beta_f$.
\item
For any composable morphisms $X\xrightarrow{f}Y\xrightarrow{g}Z$ in $\operatorname{QProj}_S$, we have
\begin{align}
\label{eq:gammacomp}
\gamma_g\cdot\gamma_f\simeq \gamma_{g\circ f}:E_X\to f_r^!g_r^!E_Z.
\end{align}
\end{enumerate}
Furthermore, if all maps $\alpha_i$ in~\eqref{eq:alphaci} and all maps $\beta_f$ in~\eqref{eq:betasm} are isomorphisms, then all maps $\gamma_f$ in~\eqref{eq:gammaqp} are isomorphisms.
\end{lemma}
\proof
We know that every morphism $f:X\to Y$ in $\operatorname{QProj}_S$ is quasi-projective, and therefore factors as 
\begin{align}
\label{eq:qpfac}
X\xrightarrow{i}P\xrightarrow{g}Y,
\end{align}
where $i$ is a closed immersion and $g$ is a smooth morphism. Given such a factorization, we define 
\begin{align}
\label{eq:gamma}
\gamma_f=\beta_g\cdot\alpha_i:E_X\to f_r^!E_Y.
\end{align}
We show that the map $\gamma_f$ defined in~\eqref{eq:gamma} is independent on the choice of the factorization~\eqref{eq:qpfac}, and is compatible with compositions in the sense of~\eqref{eq:gammacomp}.

\begin{enumerate}
\item
We first show that if $f:X\to Y$ is a closed immersion that factors as $X\xrightarrow{i}P\xrightarrow{g}Y$, where $i$ is a closed immersion and $g$ is smooth, then we have
\begin{align}
\beta_g\cdot\alpha_i=\alpha_f:E_X\to f_r^!E_Y. 
\end{align}
Indeed, we have a commutative diagram in $\operatorname{QProj}_S$
\begin{align}
\begin{split}
  \xymatrix@=11pt{
    X\ar[r]^-{} \ar@{=}[rd] \ar@/^1pc/[rr]^-{i}  & X\times_YP \ar[d]^-{g_1} \ar[r]_-{f_1} & P \ar[d]^-{g}\\
    & X \ar[r]^-{f} & Y 
  }
  \end{split}
\end{align}
where $X\xrightarrow{(id,i)}X\times_YP$ is the graph of the closed immersion $i$, $f_1$ is a closed immersion and $g_1$ is smooth. Since the square is Cartesian and $(id,i)$ is a section of $g_1$, it follows our assumptions that
\begin{align}
\label{eq:cigood}
\beta_g\cdot\alpha_i
\overset{\eqref{eq:cicomp}}{=}
\beta_g\cdot\alpha_{f_1}\cdot\alpha_{(id,i)}
\overset{\eqref{eq:cscm}}{=}
\alpha_f\cdot\beta_{g_1}\cdot\alpha_{(id,i)}
\overset{\eqref{eq:seccmp}}{=}
\alpha_f
\end{align}
which proves the claim.

We now prove that the map~\eqref{eq:gamma} is independent of the factorization. 
Suppose given two factorizations of $f$: $X\xrightarrow{i}P\xrightarrow{g}Y$ and $X\xrightarrow{i_1}P_1\xrightarrow{g_1}Y$, where $i$ and $i_1$ are closed immersions and $g$ and $g_1$ are smooth. We then have a commutative diagram in $\operatorname{QProj}_S$
\begin{align}
\begin{split}
  \xymatrix@=11pt{
    X\ar[r]^-{} \ar[rd]_-{i_1}\ar@/^1pc/[rr]^-{i}  & P\times_YP_1 \ar[d]^-{q_1} \ar[r]_-{q} & P \ar[d]^-{g}\\
    & P_1 \ar[r]^-{g_1} & Y 
  }
  \end{split}
\end{align}
where $X\xrightarrow{(i,i_1)}P\times_YP_1$ is the diagonal embedding, both $q$ and $q_1$ are smooth, and the square is Cartesian. Then it follows from the claim and our assumptions that
\begin{align}
\beta_g\cdot \alpha_i
\overset{\eqref{eq:cigood}}{=}
\beta_g\cdot \beta_q\cdot\alpha_{(i,i_1)}
\overset{\eqref{eq:smcomp}}{=}
\beta_{g_1}\cdot \beta_{q_1}\cdot\alpha_{(i,i_1)}
\overset{\eqref{eq:cigood}}{=}
\beta_{g_1}\cdot \alpha_{i_1}
\end{align}
which proves the result.

\item 
We now prove that the map~\eqref{eq:gamma} is compatible with compositions.
If $f:X\to Y$ and $g:Y\to Z$ are two composable quasi-projective morphisms in $\operatorname{QProj}_S$, then there exists a commutative diagram in $\operatorname{QProj}_S$
\begin{align}
\begin{split}
  \xymatrix@=11pt{
    X\ar[r]^-{i_1} \ar[rd]_-{f} & P \ar[d]^-{p_1} \ar[r]^-{i_3} & R \ar[d]^-{p_3}\\
    & Y \ar[rd]_-{g} \ar[r]^-{i_2} & Q \ar[d]^-{p_2}\\
    & & Z
  }
  \end{split}
\end{align}
where the horizontal maps are closed immersions and the vertical maps are smooth, and the square is Cartesian (\cite[Remarks 5.1.23]{LM}). Then it follows that
\begin{align}
\gamma_{g\circ f}
=
\beta_{p_2\circ p_3}\cdot \alpha_{i_3\circ i_1}
=
\beta_{p_2}\cdot \beta_{p_3}\cdot \alpha_{i_3}\cdot \alpha_{i_1}
\overset{\eqref{eq:cscm}}{=}
\beta_{p_2}\cdot \alpha_{i_2}\cdot \beta_{p_1}\cdot \alpha_{i_1}
=
\gamma_{g}\cdot \gamma_{f}
\end{align}
which proves~\eqref{eq:gammacomp}.
\end{enumerate}
The uniqueness of the map $\gamma_f$ in~\eqref{eq:gammaqp}, as well as the last statement, are clear from the construction above.
\endproof

\subsubsection{}
We now apply Lemma~\ref{lm:gamma} to our Gersten complex $\underline{G}(X_r,K^{ })$ defined in~\eqref{eq:def_CXrK}. Let $S$ be a scheme with a dualizing complex $K$. For every object $p:X\to S$ in $\operatorname{QProj}_S$, the complex $K^!_X=p^!K$ is a dualizing complex on $X$, and we consider the object
\begin{align}
E_X=\underline{G}(X_r,K^!_X)\in D(X_r).
\end{align}
\begin{proposition}\label{prop:compatibility}
The conditions in Lemma~\ref{lm:gamma} are satisfied if we take the map $\alpha_i$ in~\eqref{eq:alphaci} to be the map defined in~\eqref{eq:Galpha}, and the map $\beta_f$ in~\eqref{eq:betasm} to be the map defined in~\eqref{eq:Gbeta}.
\end{proposition}
\proof
We check the conditions one by one:
\begin{enumerate}
\item
Let $W\xrightarrow{k}Z\xrightarrow{i}X$ be two composable closed immersions in $\operatorname{QProj}_S$, and let $U$ be an open constructible subset of $X_r$. By construction, the map 
\begin{align}
\label{eq:cipf}
G(Z_r\cap U, i^!K)\to G(U, K)
\end{align}
in~\eqref{eq:locG} is such that the composition
\begin{align}
G(W_r\cap U,k^!i^!K^!_X)\to G(Z_r\cap U, i^!K^!_X)\to G(U, K^!_X)
\end{align}
agrees with the map associated to the closed immersion $i\circ k$. By sheafification, the map \begin{align}
\label{eq:cipfsh}
i_*\underline{G}(Z_r,i^!K^!_X)\to\underline{G}(X_r,K^!_X)
\end{align}
in~\eqref{eq:locGsh} is such that the composition
\begin{align}
i_*k_*\underline{G}(W_r\cap U,k^!i^!K^!_X)\to i_*\underline{G}(Z_r\cap U, i^!K^!_X)\to \underline{G}(U, K^!_X)
\end{align}
agrees with the map associated to the closed immersion $i\circ k$. By adjunction, this proves~\eqref{eq:cicomp}.

\item Let $V\xrightarrow{g}X\xrightarrow{f}Y$ be two composable smooth morphisms in $\operatorname{QProj}_S$, and let $U$ be an open constructible subset of $V_r$. From the description of the map~\eqref{eq:CGsm} in~\eqref{eq:cousmapdet}, we see that the map $G(f_r(U),K)\to G(U,f^*K)$ in~\eqref{eq:GUfupper} is such that the composition
\begin{align}
G(f_r(g_r(U)),K^!_Y)\to G(g_r(U),f^*K^!_Y)\to G(U,g^*f^*K^!_Y)
\end{align}
agrees with the map~\eqref{eq:GUfupper} associated to the composition $f\circ g$. This shows that the isomrophism $f_r^*\underline{G}(Y,K)\simeq \underline{G}(X,f^*K)$ in~\eqref{eq:Gsmpb} is such that the composition
\begin{align}
g_r^*f_r^*\underline{G}(Y,K^!_Y)\simeq g_r^*\underline{G}(X,f^*K^!_Y)\simeq \underline{G}(V,g^*f^*K^!_Y)
\end{align}
agrees with the map~\eqref{eq:Gsmpb} associated to the composition $f\circ g$. Since the isomorphisms~\eqref{eq:poincare} and~\eqref{eq:lcipur} are compatible with compositions, we deduce that the isomorphism $\beta_f:f_r^!\underline{G}(Y,K)\simeq\underline{G}(X,f^!K)$ in~\eqref{eq:Gbeta} satisfies~\eqref{eq:smcomp}.


\item Consider a smooth morphism $f:X\to Y$ in $\operatorname{QProj}_S$ of relative dimension $d$ with a section $i:Y\to X$. Then $i$ is a regular closed immersion, and its normal bundle $N$ agrees with $i^{-1}T_f$. In~\eqref{eq:abspur0} we defined a natural transformation $i_r^*(-)\otimes\mathbb{Z}(\operatorname{det}(N)^{-1})[-c]\to i_r^!(-)$, which induces a map in $D(X_r)$
\begin{align}
\label{eq:gysinig}
i_*i^*\underline{G}(X_r,f^*K^!_Y)
\to
\underline{G}(X_r,f^*K^!_Y)\otimes\mathbb{Z}(\operatorname{det}(T_f))[d].
\end{align}
The map~\eqref{eq:gysinig} can be described as follows: for every point $y\in Y$, one has $\mu_{f^*K^!_Y}(i(y))=\mu_{K^!_Y}(y)+d$ and $K^!_{Y,k(y)}\simeq (f^*K^!_Y\otimes\operatorname{det}(T_f))_{k(i(y))}$, and therefore for every open constructible subset $U$ of $X_r$ there is a canonical inclusion
\begin{align}
\label{eq:secinc}
\underset{\mu_K(y)=t}{\bigoplus}C(y_r\cap U, \Z({K^!_{Y,k(y)}}))
\to
\underset{\mu_{f^*K^!_Y}(x)=t+d}{\bigoplus}C(x_r\cap U, \Z({(f^*K^!_Y\otimes\operatorname{det}(T_f))_{k(x)}}))
\end{align}
sending the component of $y$ to the component of $x=i(y)$. The map~\eqref{eq:secinc} induces a map of complexes
\begin{align}
G(Y_r\cap U,K^!_Y)\to G(U,f^*K^!_Y)\otimes\mathbb{Z}(\operatorname{det}(T_f))[d]
\end{align}
whose sheafification agrees with the composition
\begin{align}
\begin{split}
i_*\underline{G}(Y_r,K^!_Y)
=
i_*i^*f^*\underline{G}(Y_r,K^!_Y)
\xrightarrow{\eqref{eq:Gsmpb}}
&i_*i^*\underline{G}(X_r,f^*K^!_Y)\\
\xrightarrow{\eqref{eq:gysinig}}
&\underline{G}(X_r,f^*K^!_Y)\otimes\mathbb{Z}(\operatorname{det}(T_f))[d].
  \end{split}
\end{align}
It follows that we have a commutative diagram
\begin{align}
\begin{split}
  \xymatrix@=12pt{
  i_*i^*f^*\underline{G}(Y_r,K^!_Y) \ar@{=}[d]^-{} \ar[r]_-{\sim}^-{\eqref{eq:Gsmpb}} & i_*i^*\underline{G}(X_r,f^*K^!_Y) \ar[r]^-{\eqref{eq:gysinig}}_-{} & \underline{G}(X_r,f^*K^!_Y)\otimes\mathbb{Z}(\operatorname{det}(T_f))[d] \ar[ld]^-{\eqref{eq:sm*!}} \\
   i_*\underline{G}(Y_r,K^!_Y) \ar[r]^-{\eqref{eq:cipfsh}} & \underline{G}(X_r,f^!K^!_Y). & 
  }
  \end{split}
\end{align}
Using adjunction and the naturality of the natural transformation~\eqref{eq:abspur0}, we deduce the following commutative diagram
\begin{align}
\begin{split}
  \xymatrix@=12pt{
  i^!\underline{G}(X_r,f^!K^!_Y) \ar[rr]_-{\sim}^-{\eqref{eq:sm*!}} & & i^!\underline{G}(X_r,f^*K^!_Y)\otimes\mathbb{Z}(\operatorname{det}(N))[d] \\
   & i^*\underline{G}(X_r,f^*K^!_Y) \ar[ru]^-{\eqref{eq:abspur0}} & i^!f^*\underline{G}(Y_r,K^!_Y)\otimes\mathbb{Z}(\operatorname{det}(N))[d] \ar[u]^-{\wr}_-{\eqref{eq:Gsmpb}} \ar[d]_-{\wr}^-{\eqref{eq:poincare}} \\
  \underline{G}(Y_r,K^!_Y) \ar@{=}[r]^-{} \ar[uu]_-{\wr}^-{\eqref{eq:Galpha}} & i^*f^*\underline{G}(Y_r,K^!_Y) \ar[u]_-{\wr}^-{\eqref{eq:Gsmpb}} \ar@{=}[r]^-{} \ar[ru]^-{\eqref{eq:abspur0}} & i^!f^!\underline{G}(Y_r,K^!_Y)
  }
  \end{split}
\end{align}
which proves~\eqref{eq:seccmp}.

\item Consider a Cartesian square of schemes in $\operatorname{QProj}_S$
\begin{align}
\begin{split}
  \xymatrix@=10pt{
    W\ar[r]^-{k} \ar[d]_-{g} & X \ar[d]^-{f} \\
    Z \ar[r]^-{i} & Y
  }
  \end{split}
\end{align}
where $i$ is a closed immersion and $f$ is smooth. Let $U$ be an open constructible subset of $X_r$. We have a commutative diagram
\begin{align}
\label{eq:diagG}
\begin{split}
  \xymatrix@=11pt{
    G(g_r(W_r\cap U),i^!K^!_Y)  \ar[d]^-{}_-{\eqref{eq:GUfupper}} & G(Z_r\cap f_r(U),i^!K^!_Y) \ar[r]^-{\eqref{eq:cipf}} \ar[l]_-{\sim} & G(f_r(U),K^!_Y) \ar[d]_-{}^-{\eqref{eq:GUfupper}} \\
    G(W_r\cap U,g^*i^!K^!_Y) \ar[r]^-{\sim}  & G(W_r\cap U,k^!f^*K^!_Y) \ar[r]^-{\eqref{eq:cipf}} & G(U,f^*K^!_Y).
  }
  \end{split}
\end{align}
Indeed, since the map~\eqref{eq:cipf} is the inclusion of a subcomplex by~\eqref{eq:locG}, the commutativity follows from
the description of the map~\eqref{eq:CGsm} in~\eqref{eq:cousmapdet}. Sheafifying the diagram~\eqref{eq:diagG} gives the following commutative diagram:
\begin{align}
\begin{split}
  \xymatrix@=12pt{
    k_*g^*\underline{G}(Y_r,i^!K^!_Y)  \ar[d]^-{\wr}_-{\eqref{eq:Gsmpb}} & f^*i_*\underline{G}(Z_r,i^!K^!_Y) \ar[r]^-{\eqref{eq:cipfsh}} \ar[l]_-{\sim} & f^*\underline{G}(Y_r,K^!_Y) \ar[d]_-{\wr}^-{\eqref{eq:Gsmpb}} \\
    k_*\underline{G}(W_r,g^*i^!K^!_Y) \ar[r]^-{\sim}  & k_*\underline{G}(W_r,k^!f^*K^!_Y) \ar[r]^-{\eqref{eq:cipfsh}} & \underline{G}(Y_rf^*K^!_Y).
  }
  \end{split}
\end{align}
Tracking back the construction of the isomorphism~\eqref{eq:Gbeta} and using adjunction, since the isomorphism~\eqref{eq:poincare} is compatible with base change, we obtain the following commutative diagram
\begin{align}
\begin{split}
  \xymatrix@=12pt{
    \underline{G}_\W(k^!f^!K^!_Y) \ar[r]^-{\eqref{eq:Galpha}}_-{\sim} \ar@{=}[d]^-{} & k^!\underline{G}_Y(f^!K^!_Y) \ar[r]^-{\eqref{eq:Gbeta}}_-{\sim} & k^!f^!\underline{G}_X(K^!_Y) \ar@{=}[d]^-{}\\
    \underline{G}_\W(g^!i^!K^!_Y) \ar[r]^-{\eqref{eq:Gbeta}}_-{\sim} & g^!\underline{G}_Z(i^!K^!_Y) \ar[r]^-{\eqref{eq:Galpha}}_-{\sim} & g^!i^!\underline{G}_X(K^!_Y)
  }
  \end{split}
\end{align}
which proves~\eqref{eq:cscm}.
\end{enumerate}
\endproof
By Lemma~\ref{lm:gamma}, Theorem \ref{th:fupper!compat} is then a consequence of Proposition~\ref{prop:compatibility}, which finishes the proof.


\section{Applications}
\subsection{Biduality}
\subsubsection{}
In this section we show that the Gersten complex $\underline{G}(X_r,K^{ })$ defined in~\eqref{eq:def_CXrK} is a dualizing object. First recall the definition:
\begin{definition}
\label{def:dualizing}
Let $X$ be a scheme. The subcategory of \textbf{constructible objects} of $D(X_r)$ is the thick subcategory $D_c(X_r)$ generated by elements of the form $f_{r!}f^!_r\mathbb{Z}$, where $f:Y\to X$ is a smooth morphism.\footnote{It is known that constructible objects in $D(X_r)$ are exactly the compact objects (\cite[Lemma 2.2]{Jin}). Another equivalent characterization given by \cite[Theorem 1.4]{Jin} states that a complex $C$ in $D(X_r)$ is constructible if and only if there is a finite stratification of $X_r$ into locally closed constructible subsets such that the restriction of $C$ to each stratum is a constant sheaf associated to a perfect complex.} We say that an object $B\in D(X_r)$ \textbf{satisfies biduality} if for any constructible object $A\in D_c(X_r)$, the canonical map $$A\to R\underline{Hom}(R\underline{Hom}(A,B),B)$$ is an isomorphism. A \textbf{dualizing object} of $D(X_r)$ is a constructible object which satisifies biduality.
\end{definition}

The following lemma is a variant of \cite[Proposition 4.4.11]{CD}:
\begin{lemma}
\label{lm:CD4.4.11}
Let $X$ be a quasi-excellent scheme and let $B\in D(X_r)$. The following statements are equivalent:
\begin{enumerate}
\item For any projective morphism $f:Y\to X$, the object $f^!_rB\in D(Y_r)$ satisfies biduality;
\item For any projective morphism $f:Y\to X$ with $Y$ regular, the  canonical map 
$\mathbb{Z}\to R\underline{Hom}(f^!_rB,f^!_rB)$ is an isomorphism.
\end{enumerate}
\end{lemma}
\proof
It is clear that the first condition implies the second, so it suffices to show the converse. By virtue of~\ref{sec:2invert}, we may assume that all schemes have characteristic $0$. Since the scheme $Y$ is quasi-excellent, we know that $D_c(Y_r)$ agrees with the thick subcategory generated by elements of the form $Rp_{r*}\mathbb{Z}$, where $p:W\to Y$ is a projective morphism with $W$ regular (\cite[Theorem 2.4.9]{BD}). Therefore it suffices to show that the following canonical map is an isomorphism:
\begin{align}
Rp_{r*}\mathbb{Z}\to R\underline{Hom}(R\underline{Hom}(Rp_{r*}\mathbb{Z},f^!_rB),f^!_rB).
\end{align}
But by assumptions we have
\begin{align}
\begin{split}
Rp_{r*}\mathbb{Z}&\simeq Rp_{r*}R\underline{Hom}(p^!_rf^!_rB,p^!_rf^!_rB)
\simeq R\underline{Hom}(Rp_{r*}p^!_rf^!_rB,f^!_rB)\\
&\simeq R\underline{Hom}(R\underline{Hom}(Rp_{r*}\mathbb{Z},f^!_rB),f^!_rB)
\end{split}
\end{align}
which finishes the proof.
\endproof

We are now ready to prove the following theorem:
\begin{theorem}
\label{th:CXKdual}
Let $X$ be an excellent scheme and let $K$ be a dualizing complex on $X$. Then the complex $\underline{G}(X_r,K^{ })\in D(X_r)$ is a dualizing object.
\end{theorem}
\proof
First we show that the complex $\underline{G}(X_r,K^{ })\in D(X_r)$ is constructible. We use noetherian induction on $X$. By working with each irreducible component, we may assume that $X$ is integral. Since $X$ is in addition quasi-excellent, there is a non-empty open immersion $j:U\to X$ such that $U$ is regular. Let $i:Z\to X$ be its reduced complement. By localization and Theorem~\ref{th:fupper!compat}, there is a distinguished triangle in $D(X_r)$
\begin{align}
Ri_{r*}\underline{G}(Z_r,i^!K)
\to
\underline{G}(X_r,K^{ })
\to
Rj_{r*}\underline{G}(U_r,j^!K)
\to
Ri_{r*}\underline{G}(Z_r,i^!K)[1].
\end{align}
By Corollary~\ref{cor:sch2.3}
, we know that $\underline{G}(U_r,j^!K)$ is constructible, and by noetherian induction $\underline{G}(Z_r,i^!K)$ is constructible. Since both $Ri_{r*}$ and $Rj_{r*}$ preserve constructible objects (\cite[Theorem 2.4.9]{BD}), this shows that $\underline{G}(X_r,K^{ })$ is constructible.

It remains to show that $\underline{G}(X_r,K^{ })$ satisfies biduality. By Lemma~\ref{lm:CD4.4.11} it suffices to show that for any projective morphism $f:Y\to X$ with $Y$ regular, the following canonical map is an isomorphism:
\begin{align}
\label{eq:Zbidualfupper}
\mathbb{Z}\to R\underline{Hom}(f^!_r\underline{G}(X_r,K^{ }),f^!_r\underline{G}(X_r,K^{ })).
\end{align}
But by Theorem~\ref{th:fupper!compat} we have $f^!_r\underline{G}(X_r,K^{ })\simeq \underline{G}(Y_r,f^!K^{ })$. Since $Y$ is regular, the map~\eqref{eq:Zbidualfupper} is an isomorphism by Corollary~\ref{cor:sch2.3}
, which finishes the proof.
\endproof

\subsection{Borel-Moore real homology theory}
\label{sec:realkun}
\subsubsection{}
Let $S$ be a regular scheme, $f:X\to S$ be a quasi-projective morphism. Then, $f \colon X\to S$ induces a map $f_r \colon X_r \to S_r$ of real schemes. If $L$ is a line bundle on $X$, we have an associated locally constant sheaf $\Z(L)$ on $X_r$ as in \cite{HWXZ}. 

\begin{definition}
\label{def:BMhom}
 Define the \textbf{Borel-Moore real homology theory} of $f_r : X_r \rightarrow S_r$ as the hypercohomology
		\[ H^{\mathrm{BM}}_n(X_r, L) : =  \mathbb{H}^{-n}  \underline{G}\big( X_r , f^! K_S  \otimes  L \big) \]	
\end{definition}
\begin{remark}\label{rmk:BMiso}
We have the following identifications
	\[H_n^{\mathrm{BM}}(X_r, L) \cong \mathbb{H}^{-n} (f_r^!\Z_{S_r} \otimes \mathbb{Z}(L)) \]
	where $\underline{G}(X_r, f^!K_S) \simeq f_r^!\underline{G}(S, K_S) $ is a flasque resolution of $ f_r^!\Z_{S_r}$ by Corollary~\ref{cor:sch2.3}, 
	\ref{num:Gerlb} and Theorem \ref{th:fupper!compat}. Here, $\Z_{S_r}$ is the constant sheaf on $S_r$.\ If $S = \Spec (R)$ and $R$ be a real closed field, we have
	\[H_n^{\mathrm{BM}}(X_r, L)  \cong H^{\mathrm{BM}}_n(X(R),\mathbb{Z}(L))\]
	where the right hand side is the semi-algebraic Borel-Moore homology (\cite[III \S2]{Del}), and $X(R)$ is the semialgebraic space defined by the $R$-points of $X$ (\cite[I Example 1.1]{Del}). This can be seen by \cite[III 12.10]{Del}. It is worth to mention that, in the special case $S = \Spec (\R)$, we get the classical Borel-Moore homology (\cite{BM}, \cite{Iver}, \cite{Bre}). 
\end{remark}

\begin{definition} 
Define the \textbf{$\mathrm{I}$-homology theory} as 
	\begin{align}
	H_{n}(X,\mathrm{I}^\infty(L))=
		H^{-n}(X,\mathrm{I}^\infty, f^!K_S \otimes L). 
\end{align}
\end{definition}

\begin{proposition}\label{coro:signature-borel-moore}
	Let $X$ be a quasi-projective $\R$-scheme. Let $L$ be a line bundle on $X$. Then the signature morphism 
	\[ \sign: H_n (X, \I^\infty(L)) \xrightarrow{ \cong } H^{\mathrm{BM}}_n(X(\R), \mathbb{Z}(L))  \]
	is an  isomorphism for all $n \in \Z $, where the right-hand side is the classical Borel-Moore homology in topology (\cite{BM}).
\end{proposition}
\begin{proof}
	Note that by Proposition \ref{Prop: I-cohomology-Real}, the signature map induces the following isomorphism 
	$$H_n (X, \I^\infty(L)) \cong H^{\mathrm{BM}}_n(X_r, L)$$ 
	The result follows from Remark \ref{rmk:BMiso}. 
\end{proof}

\begin{corollary}
\label{cor:kun_BMr}
Let $X,Y$ be two quasi-projective $\mathbb{R}$-schemes. Let $L$ (respectively $L'$) be a line bundle on $X$ (respectively $Y$). Then there is a split short exact sequence
\begin{align}
\begin{split}
0
&\to
\bigoplus\limits_{p+q=n}H^{\mathrm{BM}}_p(X,L)\otimes_\mathbb{Z}H^{\mathrm{BM}}_q(Y,L')
\to
H^{\mathrm{BM}}_n(X\times_\mathbb{R}Y,L\boxtimes_\mathbb{R}L')\\
&\to
\bigoplus\limits_{p+q=n-1}\operatorname{Tor}(H^{\mathrm{BM}}_p(X,L),H^{\mathrm{BM}}_q(Y,L'))
\to
0
\end{split}
\end{align}
(see~\ref{sec:outtens} for the notation $L\boxtimes_{\mathbb{R}}L'$).
\end{corollary}
\proof
By \cite[III 11.10]{Del} we are reduced to topological Borel-Moore homology (\cite[V.3]{Bre}). Since $(X\times_\mathbb{R}Y)(\mathbb{R})\simeq X(\mathbb{R})\times Y(\mathbb{R})$, the result then follows from \cite[V.14.1]{Bre}: the $hlc$ condition on $X(\mathbb{R})$ and $Y(\mathbb{R})$ is satisfied since both spaces are locally contractible.
\endproof

\begin{remark}
\begin{enumerate}
\item The $\operatorname{Tor}$-terms are non-zero in general, unless we work with rational coefficients. For example, for any pair of coprime integers $(p,q)$ there exists a smooth $\mathbb{R}$-scheme whose set of $\mathbb{R}$-points is homeomorphic to the $3$-dimensional lens space $L(p,q)$, and consequently its Borel-Moore homology, which agrees with singular homology since lens spaces are compact, has $p$-torsion.
\item %
We expect the K\"unneth formula to hold over any real closed field, 
which we will investigate in a forthcoming work. However, when the base field $k$ is not real closed, the situation is a priori more complicated as the real scheme of a fiber product over $k$ is not necessarily the Cartesian product, and one also needs to take care of the geometric components.
\item Since the complex $\underline{G}(X_r,K^{ })$ is a complex of acyclic sheaves, the hypercohomology groups $H^{\mathrm{BM}}_{n}(X,L)$ can be computed as the cohomology groups of the complex of global sections. By \cite[Theorem 8.10]{DFJK} we know, after inverting $2$, that the groups $H^{\mathrm{BM}}_{n}(X,L)$ agree with the twisted $\mathbb{A}^1$-bivariant groups in the negative part of the $\mathbb{A}^1$-derived category $D^{\mathbb{A}^1}$. In a forthcoming work we will study a link of these groups with Chow-Witt groups.
\end{enumerate}
\end{remark}








\end{document}